\documentclass[a4paper, 11pt, leqno]{amsart}

\usepackage{amsmath,amssymb,amsthm}
\usepackage{graphicx}
\usepackage{cite,mathrsfs}

\newtheorem{theorem}{Theorem}[section]
\newtheorem{lemma}[theorem]{Lemma}
\newtheorem{prop}[theorem]{Proposition}
\newtheorem{corollary}[theorem]{Corollary}
\newtheorem{claim}[theorem]{Claim}

\theoremstyle{definition}
\newtheorem{definition}[theorem]{Definition}
\newtheorem{example}[theorem]{Example}

\theoremstyle{remark}
\newtheorem{remark}[theorem]{Remark}

\numberwithin{equation}{section}
\allowdisplaybreaks
\DeclareMathOperator*{\Res}{Res}
\DeclareMathOperator{\Aut}{Aut}

\def\id{\mathrm{id}}

\def\abs#1{\lvert#1\rvert}
\long\def\comment#1{}
\def\Lh#1{[#1]}

\def\VN#1{\overset{\rlap{$#1$}}{\V}}% vertex with name
\def\EB#1#2{\rlap{\raisebox{-0.3ex}{#2}}\raisebox{0.3ex}{#1}}% double egde for type-B or C
\def\V{{\setlength{\unitlength}{1pt}\begin{picture}(4,4)\put(2,2){\circle{4}}\end{picture}}}% vertex
\def\E{{\setlength{\unitlength}{1pt}\begin{picture}(38,2)\put(0,2){\line(1,0){38}}\end{picture}}}% edge
\def\EO{{\setlength{\unitlength}{1pt}\begin{picture}(38,6)\put(0,2){\line(1,0){16}}\put(22,2){\line(1,0){16}}\put(16,4){\oval(4,4)[l]}\put(16,0){\oval(4,4)[r]}\put(22,4){\oval(4,4)[l]}\put(22,0){\oval(4,4)[r]}\end{picture}}}% omission
\allowdisplaybreaks

\begin{document}
%\hfill\texttt{\jobname.tex}\qquad\today

\title[Zeta-functions of root systems and Poincar{\'e} polynomials]{Zeta-functions of root systems and Poincar{\'e} polynomials of Weyl groups}

\author{Yasushi Komori}
\address{Y. Komori: Department of Mathematics, Rikkyo University, Nishi-Ikebukuro, Toshima-ku, Tokyo 171-8501, Japan}
\email{komori@rikkyo.ac.jp}

\author{Kohji Matsumoto}
\address{K. Matsumoto: Graduate School of Mathematics, Nagoya University, Chikusa-\
ku, Nagoya 464-8602, Japan}
\email{kohjimat@math.nagoya-u.ac.jp}

\author{Hirofumi Tsumura}
\address{{H.\,Tsumura:} Department of Mathematics and Information Sciences, Tokyo \
Metropolitan University, 1-1, Minami-Ohsawa, Hachioji, Tokyo 192-0397, Japan}
\email{tsumura@tmu.ac.jp}


\begin{abstract}
We consider a certain linear combination $S(\mathbf{s},\mathbf{y};I;\Delta)$
of zeta-functions of root systems, where $\Delta$ is a root system of rank $r$ and
$I\subset\{1,2,\ldots,r\}$.    Showing two different expressions of
$S(\mathbf{s},\mathbf{y};I;\Delta)$, we find 
that a certain signed sum of zeta-functions of root systems is
equal to a sum involving Bernoulli functions of root systems.    
This identity gives a non-trivial functional relation among zeta-functions of root
systems, if the signed sum does not identically vanish.     
This is a genralization of the authors' previous result proved in \cite{KMTLondon},
in the case when $I=\emptyset$.    We present several
explicit examples of such functional relations.     A criterion of the non-vanishing
of the signed sum, in terms of Poincar{\'e} polynomials of associated Weyl groups,
is given.     Moreover we prove a certain converse theorem, which implies that the
generating function for the case $I=\emptyset$ essentially knows all information on
generating functions for general $I$.
\end{abstract}

\maketitle

\baselineskip 16pt

%%%%%%%%%%%%%%%%%%%%%%%%%%%%%%%%%%%%%%%%%%%%%%%%%%%%%%%%%%%%%%%%%%%%%%%%%%%%%%%%%%%%%
\section{Introduction}
%%%%%%%%%%%%%%%%%%%%%%%%%%%%%%%%%%%%%%%%%%%%%%%%%%%%%%%%%%%%%%%%%%%%%%%%%%%%%%%%%%%%%

Let $\mathbb{N}$ be the set of positive integers, $\mathbb{N}_0$ the set of
non-negative integers, $\mathbb{Z}$ the set of rational integers, $\mathbb{R}$ the set of real numbers, and $\mathbb{C}$ the set of complex numbers. 
For any set $S$, denote by $|S|$ the cardinality of $S$.

Let $V$ be an $r$-dimensional real vector space equipped with an inner product
$\langle\cdot,\cdot\rangle$.   The dual space $V^*$ is identified with $V$
via this inner product.    Let $\Delta$ be a finite reduced root system in $V$
and $\Psi=\{\alpha_1,\ldots,\alpha_r\}$ its fundamental system.   Let 
$\Delta_+$ and $\Delta_-$ be the sets of all positive roots and negative roots,
respectively: $\Delta=\Delta_+\coprod\Delta_-$.
We denote by $\alpha^{\vee}$ the coroot associated with a root $\alpha$.
Let $\Lambda=\{\lambda_1,\ldots,\lambda_r\}$ be the set of fundamental weights
defined by $\langle\alpha_i^{\vee},\lambda_j\rangle=\delta_{ij}$ (Kronecker's delta).
Let $Q^{\vee}$ be the coroot lattice,
$P$ the weight lattice, $P_+$ the set of integral dominant weights, and
$P_{++}$ the set of integral strongly dominant weights, respectively, defined 
by
$$
Q^{\vee}=\bigoplus_{i=1}^r \mathbb{Z}\alpha_i^{\vee},\quad
P=\bigoplus_{i=1}^r\mathbb{Z}\lambda_i,\quad
P_+=\bigoplus_{i=1}^r\mathbb{N}_0\lambda_i,\quad
P_{++}=\bigoplus_{i=1}^r\mathbb{N}\lambda_i.
$$

Let $\mathbf{y}\in V$, and $\mathbf{s}=(s_{\alpha})_{\alpha\in\Delta_+}
\in\mathbb{C}^{|\Delta_+|}$.
The zeta-function of the root system $\Delta$ is defined by
\begin{equation}
  \label{1-1}
  \zeta_r(\mathbf{s},\mathbf{y};\Delta)=\sum_{\lambda\in P_{++}}
  e^{2\pi\sqrt{-1}\langle \mathbf{y},\lambda\rangle}
  \prod_{\alpha\in\Delta_+}
  \frac{1}{\langle\alpha^\vee,\lambda\rangle^{s_\alpha}}.
\end{equation}

This function was introduced and has been studied by the authors in
\cite{KMTKyushu} \cite{KMTWitten2} \cite{KMTLondon} \cite{KMTNicchuu}
\cite{KMTWitten4} \cite{KMTWitten3} \cite{KMTWitten5} and \cite{MTFourier}.

Let $\mathfrak{g}$ be a complex semisimple Lie algebra.   If
$\Delta=\Delta(\mathfrak{g})$ is the root system associated with
$\mathfrak{g}$, and $s_{\alpha}=s$ for all $\alpha\in\Delta_+$, then
$\zeta_r((s,s,\ldots,s),\mathbf{0};\Delta)$ is essentially equal to the Witten
zeta-function of $\mathfrak{g}$ studied by Witten \cite{Witten} and Zagier
\cite{Zagier}, up to a certain simple factor (see \cite[(1.7)]{KMTWitten2}).
It is well known that simple Lie algebras are classified into seven types; we denote 
them by $X_r$, where $X=A, B, C, D, E, F, G$ and $r$ denotes its rank.    
When $\mathfrak{g}$ is of
type $X_r$, we frequently write its root system as $\Delta(X_r)$, and
its zeta-function as 
$\zeta_r(\mathbf{s}, \mathbf{y};X_r)$. 
In particular, $\zeta_1(s,0;A_1)=\zeta(s)$, the classical Riemann zeta-function,
and
$$
\zeta_2((s_1,s_2,s_3),\mathbf{0};A_2)=\sum_{m_1=1}^{\infty}\sum_{m_2=1}^{\infty}m_1^{-s_1}
m_2^{-s_2}(m_1+m_2)^{-s_3},
$$
which is sometimes called the (Mordell-)Tornheim double sum \cite{Tornheim}.

On the other hand, we can see that the Euler-Zagier $r$-ple zeta-function
$$
\zeta_{EZ,r}(s_1,\ldots,s_r)=\sum_{m_1=1}^{\infty}\cdot\sum_{m_r=1}^{\infty}m_1^{-s_1}
(m_1+m_2)^{-s_2}\cdots (m_1+\cdots+m_r)^{-s_r}
$$
may be regarded as a special case of zeta-functions of root systems of type $A_r$
(see \cite{KMTMathZ}), or of type $C_r$ (see \cite{KMTFACM}).

Therefore we can say that the notion of zeta-functions of root systems gives a 
unification of two important class of multiple zeta-functions, of Witten and of
Euler and Zagier.

A lot of relations among special values (at integer points) of Euler-Zagier multiple 
zeta-functions are known.
Are there any functional relations which interpolate those relations?    
This question was raised, around 2000, by the second-named author (cf. \cite{Mat06}).
Needless to say,
harmonic product formulas such as
$$
\zeta(s_1)\zeta(s_2)=\zeta_{EZ,2}(s_1,s_2)+\zeta_{EZ,2}(s_2,s_1)+\zeta(s_1+s_2)
$$
are valid not only at integer points, but also at any other complex values of $s_1$ and
$s_2$, so these give an answer.    But what else?
So far, no other functional relations has been discovered among Euler-Zagier multiple
zeta-functions (except for a kind of functional equation for the case $r=2$ discovered by the
second-named author \cite{Mat04}), and in fact, a kind of negative answer was obtained
recently by Ikeda and Matsuoka \cite{IMprep}.

However, if we extend the range of the search, we may find such functional relations.
The first example is a functional relation between $\zeta(s)$ and 
$\zeta_2((s_1,s_2,s_3),\mathbf{0};A_2)$ discovered by the third-named author \cite{Tsu07},
which interpolates certain value relations among $\zeta(k)$ and $\zeta_{EZ,2}(k_1,k_2)$
($k,k_1,k_2\in\mathbb{N}$).    After this discovery, various other functional relations
among zeta-functions of root systems have been reported (the aforementioned papers of
the authors, Nakamura \cite{Nak06} \cite{Nak08}, Zhou et al. \cite{ZBC}, 
Onodera \cite{Onod14}, and 
Ikeda and Matsuoka \cite{IM}).

On the other hand, the structural background of the existence of those functional
relations has been studied in \cite{KMTWitten3}, \cite{KMTLondon}.
The present paper is a continuation of these two papers.
The main actor of the present paper is the ``Weyl-group-symmetric'' linear combination
$S(\mathbf{s},\mathbf{y};I;\Delta)$ of zeta-functions of root systems defined 
by \eqref{2-1} below.
This $S(\mathbf{s},\mathbf{y};I;\Delta)$ has two different expressions:
(i) It is a signed sum of zeta-functions of root systems, 
and on the other hand, (ii) it can be expressed in terms of certain
generalization of Bernoulli functions $P(\mathbf{k},\mathbf{y},\lambda;I;\Delta)$;
the exact form of these facts will be stated in Section \ref{sec2}
(\eqref{sahen} and Theorem \ref{thm:main1}). 

Combining these two expressions, (if they do not identically vanish) we can obtain
certain functional relations among zeta-functions of root systems.    
We will state several
explicit forms of such functional relations.
%will be discussed in Sections \ref{sec3}, 
%\ref{sec4}, and \ref{sec7}.
However, since the expression (i) is a signed sum, there is the possibility that it
vanishes identically.    Whether it vanishes or not can be seen by observing the
associated Poincar{\'e} polynomials.    This is another main theme of the present paper.   
%This matter will be studied in
%Sections \ref{sec5} and \ref{sec6}.

%%%%%%%%%%%%%%%%%%%%%%%%%%%%%%%%%%%%%%%%%%%%%%%%%%%%%%%%%%%%%%%%%%%%%%%%%%%%%%%%%%%%%%%%
\section{Fundamental formulas}\label{sec2}
%%%%%%%%%%%%%%%%%%%%%%%%%%%%%%%%%%%%%%%%%%%%%%%%%%%%%%%%%%%%%%%%%%%%%%%%%%%%%%%%%%%%%%%%

Let $I\subset\{1,2,\ldots,r\}$, and $\Psi_I=\{\alpha_i \;|\; i\in I\}\subset \Psi$.
Let $V_I$ be the subspace of $V$ spanned by $\Psi_I$.    Then $\Delta_I=\Delta\cap V_I$
is the root system in $V_I$ whose fundamental system is $\Psi_I$.
For $\Delta_I$, we denote the corresponding coroot lattice, weight lattice etc.\ by 
$Q_I^{\vee}=\bigoplus_{i\in I}\mathbb{Z}\alpha_i^{\vee}$,
$P_I=\bigoplus_{i\in I}\mathbb{Z}\lambda_i$ etc. 
Let $\iota:Q_I^{\vee}\to Q^{\vee}$ be the natural embedding, 
and $\iota^*:P\to P_I$ the projection induced from $\iota$; that is, for 
$\lambda\in P$, $\iota^*(\lambda)$ is defined as a unique element of $P_I$ satisfying
$\langle\iota(q),\lambda\rangle=\langle q,\iota^*(\lambda)\rangle$ for all
$q\in Q_I^{\vee}$.

Let ${\rm Aut}(\Delta)$ be the subgroup of ${\rm GL}(V)$, consisting of all automorphisms
which stabilizes $\Delta$.
Let $\sigma_{\alpha}\in{\rm Aut}(\Delta)$ be the reflection with respect to $\alpha$, and 
denote by $W=W(\Delta)$ the Weyl
group of $\Delta$, namely the group generated by $\{\sigma_{i}\;|\;1\leq i\leq r\}$,
where $\sigma_i=\sigma_{\alpha_i}$.
This is a normal subgroup of ${\rm Aut}(\Delta)$.
For $w\in W$, we put $\Delta_w=\Delta_+\cap w^{-1}\Delta_-$.
Let $W_I$ be the subgroup of $W$ generated by all the reflections associated with the elements
in $\Psi_I$, and
$W^I=\{w\in W\;|\; \Delta_{I+}^{\vee}\subset w\Delta_+^{\vee}\}$.

The fundamental Weyl chamber is defined by
$$
C=\{v\in V\;|\; \langle\alpha_i^{\vee},v\rangle\geq 0\;{\rm for}\;1\leq i\leq r\}.
$$
Then $W$ acts on the set of Weyl chambers $\{wC\;|\;w\in W\}$ simply transitively.
For any subset $A\subset\Delta$, 
let $H_{A^{\vee}}$ be the set of all
$v\in V$ which satisfies $\langle\alpha^{\vee},v\rangle= 0$ for some $\alpha\in A$.  
In particular, $H_{\Delta^{\vee}}$ is the set of all walls of Weyl chambers.

Now define
\begin{align}\label{2-1}
S(\mathbf{s},\mathbf{y};I;\Delta)=\sum_{\lambda\in \iota^{*-1}(P_{I+})\setminus 
  H_{\Delta^{\vee}}}
  e^{2\pi\sqrt{-1}\langle \mathbf{y},\lambda\rangle}
  \prod_{\alpha\in\Delta_+}
  \frac{1}{\langle\alpha^\vee,\lambda\rangle^{s_\alpha}}.
  \end{align}
This sum was first introduced in \cite[(110)]{KMTWitten3}.
(Note that \cite{KMTWitten3} was published later than \cite{KMTLondon}, but was
written earlier, already in 2007).
In \cite[Theorems 5 and 6]{KMTWitten3}, we showed
\begin{align}\label{sahen}
S(\mathbf{s},\mathbf{y};I;\Delta)
    =
    \sum_{w\in W^I}
    \Bigl(\prod_{\alpha\in\Delta_{w^{-1}}}(-1)^{-s_{\alpha}}\Bigr)
    \zeta_r(w^{-1}\mathbf{s},w^{-1}\mathbf{y};\Delta).
\end{align}
In the above statement, the action of $W$ to $\mathbf{s}$ is
defined by $(w\mathbf{s})_{\alpha}=s_{w^{-1}\alpha}$ for $w\in W$ with the convention 
that, if
$\alpha\in\Delta_-$, then we understand that $s_{\alpha}=s_{-\alpha}$.

We also proved in \cite[Theorems 5 and 6]{KMTWitten3} a certain multiple integral 
expression of $S(\mathbf{s},\mathbf{y};I;\Delta)$.
When $I=\emptyset$,
we further observed that, in some special cases, those integrals may be regarded as
generalizations of classical (Seki-)Bernoulli functions, which we denoted by
$P(\mathbf{k},\mathbf{y};\Delta)$ (which is actually the special case $\lambda=\mathbf{0}$,
$I=\emptyset$ of
$P(\mathbf{k},\mathbf{y},\lambda;I;\Delta)$ defined later in \eqref{eq:def_F}).
We gave multiple integral expressions of generating functions of $P(\mathbf{k},\mathbf{y};\Delta)$
(see \cite[Theorem 7]{KMTWitten3}), but it requires extremely huge task if we want to
calculate that expression more explicitly.
This situation was improved in \cite{KMTLondon}, in which more accessible expressions were given 
(see \cite[Theorem 4.1]{KMTLondon}).

However, these are the results in the case $I=\emptyset$.    In order to develop the full
theory of functional relations, it is necessary to obtain the analogous results for general
$I$.    This is the first aim of the present paper.    To state the results, we need some
more notations.

Let 
$\Delta^*=\Delta_+\setminus\Delta_{I+}$ and
$d=\abs{I^c}$.
We may find $\mathbf{V}_I=\{\gamma_1,\ldots,\gamma_d\}\subset \Delta^*$
such that $\mathbf{V}=\mathbf{V}_I\cup\Psi_I$ becomes a basis of $V$.
Let $\mathscr{V}_I=\mathscr{V}(\Delta^*)$ be the set of all such bases.
In particular, $\mathscr{V}=\mathscr{V}_{\emptyset}$ be the set of all
linearly independent subsets
$\mathbf{V}=\{\beta_1,\ldots,\beta_r\}\subset\Delta_+$.

For $\mathbf{V}\in\mathscr{V}_I$,
the lattice $L(\mathbf{V}^{\vee})=\bigoplus_{\beta\in\mathbf{V}}\mathbb{Z}\beta^{\vee}$
is a sublattice of $Q^{\vee}$.
Let $\{\mu^{\mathbf{V}}_\gamma\}_{\gamma\in\mathbf{V}}$ 
be the dual basis of $\mathbf{V}^\vee=\mathbf{V}_I^\vee\cup\Psi_I^\vee$,
namely
$\langle \gamma_k^{\vee},\mu^{\mathbf{V}}_{\gamma_l}\rangle=\delta_{kl}$,
$\langle \alpha_i^{\vee},\mu^{\mathbf{V}}_{\alpha_j}\rangle=\delta_{ij}$, and
$\langle\gamma_k^{\vee}, \mu^{\mathbf{V}}_{\alpha_i}\rangle
=\langle\alpha_i^{\vee}, \mu^{\mathbf{V}}_{\gamma_k}\rangle=0$.
Let $p_{\mathbf{V}_I^\perp}$ 
be the projection defined by
\begin{equation}
  \label{eq:proj}
  p_{\mathbf{V}_I^\perp}(v)=
  v-\sum_{\gamma\in\mathbf{V}_I}\mu^{\mathbf{V}}_\gamma\langle\gamma^\vee,v\rangle
  =
  \sum_{\alpha\in\Psi_I}\mu^{\mathbf{V}}_{\alpha}\langle\alpha^\vee,v\rangle,
\end{equation}
for $v\in V$.  
(The second equality can be easily seen by expressing 
$v=\sum_k a_k\mu^{\mathbf{V}}_{\gamma_k}+\sum_i b_i\mu^{\mathbf{V}}_{\alpha_i}$.)
It is to be noted that the projection
  $p_{\mathbf{V}_I^\perp}$ depends only on $\mathbf{V}_I$ in the following sense.
\begin{lemma}
\label{lm:indept_bases}
  For any linearly independent subset $\Phi_I=\{\beta_1,\ldots,\beta_{|I|}\}\subset\Delta_{I+}$ and $\mathbf{U}=\mathbf{V}_I\cup\Phi_I$, we have
  \begin{equation}
    p_{\mathbf{V}_I^\perp}(v)=
    v-\sum_{\gamma\in\mathbf{V}_I}\mu^{\mathbf{U}}_\gamma\langle\gamma^\vee,v\rangle
    =
    \sum_{\beta\in\Phi_I}\mu^{\mathbf{U}}_{\beta}\langle\beta^\vee,v\rangle.
  \end{equation}
\end{lemma}
\begin{proof}
  Put $u=\sum_{\beta\in\Phi_I}\mu^{\mathbf{U}}_{\beta}\langle\beta^\vee,v\rangle$
 and
we show
  $p_{\mathbf{V}_I^\perp}(v)=u$.    It is enough to check that 
$\langle\gamma^\vee,p_{\mathbf{V}_I^\perp}(v)\rangle=\langle\gamma^{\vee},u\rangle$
for all $\gamma\in\mathbf{U}$.
  For $\gamma\in\mathbf{V}_I$, we have
  \begin{align*}
    \langle\gamma^\vee,p_{\mathbf{V}_I^\perp}(v)\rangle
    &=                                                     
      \Bigl\langle\gamma^\vee,v-\sum_{\gamma_1\in\mathbf{V}_I}\mu^{\mathbf{V}}_{\gamma_1}
      \langle\gamma_1^{\vee},v\rangle\Bigr\rangle
      =\langle\gamma^\vee,v\rangle-\langle\gamma^\vee,v\rangle=0,
    \\
    \langle\gamma^\vee,u\rangle
    &=
\Bigl\langle\gamma^\vee,v-\sum_{\gamma_1\in\mathbf{V}_I}\mu^{\mathbf{U}}_{\gamma_1}\langle
      \gamma_1^\vee,v\rangle\Bigr\rangle
      =\langle\gamma^\vee,v\rangle-\langle\gamma^\vee,v\rangle=0,
  \end{align*}
  because $\mathbf{V}_I\subset\mathbf{V},\mathbf{U}$.   For $\beta\in\Phi_I$,
  we have
  \begin{align*} 
\langle\beta^\vee,u\rangle
=\sum_{\beta_1\in\Phi_I}\Bigl\langle\beta^{\vee},\mu_{\beta_1}^{\mathbf{U}}
\langle\beta_1^{\vee},v\rangle\Bigr\rangle
=\sum_{\beta_1\in\Phi_I}\langle\beta_1^{\vee},v\rangle \langle\beta^{\vee},\mu_{\beta_1}^{\mathbf{U}}\rangle
=\langle\beta^{\vee},v\rangle,
\end{align*}
while
  by writing $\beta^\vee=\sum_{\alpha\in\Psi_I}a_\alpha\alpha^\vee$, we have
  \begin{align*}
    \langle\beta^\vee,p_{\mathbf{V}_I^\perp}(v)\rangle &=
    \sum_{\alpha_1\in\Psi_I}\langle\beta^\vee,\mu^{\mathbf{V}}_{\alpha_1}\rangle\langle\alpha_1^\vee,v\rangle
    = \sum_{\alpha\in\Psi_I}a_\alpha\langle\alpha^\vee,v\rangle\\
    &=\Bigl\langle\sum_{\alpha\in\Psi_I}a_\alpha \alpha^\vee,v\Bigr\rangle
    =\langle\beta^\vee,v\rangle.
  \end{align*}
\end{proof}

Next we introduce a generalization of the notion of ``fractional part'' of real numbers.
Let $\mathscr{V}$ be the set of linearly independent subsets
$\mathbf{V}=\{\beta_1,\ldots,\beta_r\}\subset\Delta_+$.
Also, let $\mathscr{R}$ be the set of all linearly independent subsets
$\mathbf{R}=\{\beta_1,\ldots,\beta_{r-1}\}\subset\Delta$, and let
$\mathfrak{H}_{\mathbf{R}^{\vee}}=\bigoplus_{i=1}^{r-1}\mathbb{R}\beta_i^{\vee}$
be the hyperplane passing through $\mathbf{R}^{\vee}\cup\{0\}$.    We fix a non-zero vector
$$
\phi\in V\setminus\bigcup_{\mathbf{R}\in\mathscr{R}}\mathfrak{H}_{\mathbf{R}^{\vee}}.
$$
Then $\langle\phi,\mu_{\beta}^{\mathbf{V}}\rangle\neq 0$ for all 
$\mathbf{V}\in\mathscr{V}$ and $\beta\in\mathbf{V}$.
For $\mathbf{y}\in V$, $\mathbf{V}\in\mathscr{V}$ and $\beta\in\mathbf{V}$, we define
\begin{align}
\{\mathbf{y}\}_{\mathbf{V},\beta}=\left\{
   \begin{array}{ll}
   \{\langle\mathbf{y},\mu_{\beta}^{\mathbf{V}}\rangle\},& 
        (\langle\phi,\mu_{\beta}^{\mathbf{V}}\rangle>0),\\
   1-\{-\langle\mathbf{y},\mu_{\beta}^{\mathbf{V}}\rangle\},
        & (\langle\phi,\mu_{\beta}^{\mathbf{V}}\rangle<0),
   \end{array}
   \right.
\end{align}
where $\{\cdot\}$ on the right-hand sides denotes the usual fractional part of real
numbers.

Using these notions, we now define Bernoulli functions of the root system $\Delta$
associated with $I$ and their generating functions.

\begin{definition}
For $\mathbf{t}_I=(t_{\alpha})_{\alpha\in\Delta^*}\in\mathbb{C}^{|\Delta^*|}$ and
$\lambda\in P_I$, let
\begin{align}
  \label{eq:exp_F}
 & F(\mathbf{t}_I,\mathbf{y},\lambda;I;\Delta)\\
  &=
  \sum_{\mathbf{V}\in\mathscr{V}_I}
  \left(\prod_{\gamma\in \Delta^*\setminus\mathbf{V}_I}
  \frac{t_\gamma}
  {t_\gamma-\sum_{\beta\in\mathbf{V}_I}
    t_\beta\langle\gamma^\vee,\mu^{\mathbf{V}}_\beta\rangle
    -2\pi\sqrt{-1}
    \langle \gamma^\vee,p_{\mathbf{V}_I^\perp}(\lambda)\rangle}\right)\notag
  \\
  &\qquad\times
  \frac{1}{\abs{Q^\vee/L(\mathbf{V}^\vee)}}
  \sum_{q\in Q^\vee/L(\mathbf{V}^\vee)}
  \exp(2\pi\sqrt{-1}\langle \mathbf{y}+q,p_{\mathbf{V}_I^\perp}(\lambda)\rangle)\notag\\
  &\qquad\times
  \prod_{\beta\in\mathbf{V}_I}
  \frac{t_\beta\exp
    (t_\beta
    \{\mathbf{y}+q\}_{\mathbf{V},\beta})}{e^{t_\beta}-1},\notag
\end{align}
and define {\it Bernoulli functions} $P(\mathbf{k},\mathbf{y},\lambda;I;\Delta)$ 
{\it of the root system $\Delta$
associated with $I$} by the expansion
\begin{equation}
\label{eq:def_F}
  F(\mathbf{t}_I,\mathbf{y},\lambda;I;\Delta)=
  \sum_{\mathbf{k}\in \mathbb{N}_0^{\abs{\Delta^*}}}P(\mathbf{k},\mathbf{y},\lambda;I;\Delta)
  \prod_{\alpha\in \Delta^*}
  \frac{t_\alpha^{k_\alpha}}{k_\alpha!}.
\end{equation}
\end{definition}

The fundamental formula in our theory is the following theorem.

\begin{theorem}
\label{thm:main1}
Let $s_\alpha=k_\alpha\in\mathbb{N}$ for $\alpha\in \Delta^*$ and
$s_\alpha\in\mathbb{C}$ for $\alpha\in\Delta_{I+}$. 
We assume

%{\rm (\#)} If $\alpha$ is orthogonal to all other elements of $\Delta^*$, then the corresponding $k_{\alpha}\geq 2$.  

{\rm (\#)} If $\alpha$ belongs to an irreducible component of type $A_1$,
 then the corresponding $k_{\alpha}\geq 2$.  

Then we have
\begin{equation}
  \label{eq:func_eq}
  \begin{split}
    &S(\mathbf{s},\mathbf{y};I;\Delta)
%    =
%    \sum_{w\in W^I}
%    \Bigl(\prod_{\alpha\in\Delta_{w^{-1}}}(-1)^{-s_{\alpha}}\Bigr)
%    \zeta_r(w^{-1}\mathbf{s},w^{-1}\mathbf{y};\Delta)
    \\
    &\qquad=(-1)^{\abs{\Delta^*}}
    \biggl(\prod_{\alpha\in \Delta^*}
    \frac{(2\pi\sqrt{-1})^{k_\alpha}}{k_\alpha!}\biggr)
    \sum_{\lambda\in P_{I++}}
    \biggr(
    \prod_{\alpha\in\Delta_{I+}}
    \frac{1}{\langle\alpha^\vee,\lambda\rangle^{s_\alpha}}
    \biggl)
    P(\mathbf{k},\mathbf{y},\lambda;I;\Delta).
  \end{split}
\end{equation}
\end{theorem}

In the case $I=\emptyset$, clearly $\Psi_I=\emptyset$, $\mathbf{V}_I=\mathbf{V}$, 
$\mathscr{V}_I=\mathscr{V}$, $\Delta^*=\Delta_+$, $P_I=\{\mathbf{0}\}$, and hence the only
possible $\lambda$ is $\lambda=\mathbf{0}$.
Write $\mathbf{t}=\mathbf{t}_{\emptyset}=(t_{\alpha})_{\alpha\in\Delta_+}$.
Then we see that 
\begin{equation}
  \begin{split}
    F(\mathbf{t},\mathbf{y};\Delta)
    &=
    F(\mathbf{t},\mathbf{y};\mathbf{0};\emptyset;\Delta)
    \\
    &=
    \sum_{\mathbf{V}\in\mathscr{V}}
    \left(\prod_{\gamma\in \Delta_+\setminus\mathbf{V}}
    \frac{t_\gamma}
    {t_\gamma-\sum_{\beta\in\mathbf{V}}
      t_\beta\langle\gamma^\vee,\mu^{\mathbf{V}}_\beta\rangle}\right)
    \\
    &\qquad\times
    \frac{1}{\abs{Q^\vee/L(\mathbf{V}^\vee)}}
    \sum_{q\in Q^\vee/L(\mathbf{V}^\vee)}
    \prod_{\gamma\in\mathbf{V}}
    \frac{t_\gamma\exp
      (t_\gamma\{\mathbf{y}+q\}_{\mathbf{V},\gamma})}{e^{t_\gamma}-1},
  \end{split}
\end{equation}
which is exactly equal to \cite[(4.3)]{KMTLondon}.
Therefore Theorem \ref{thm:main1} is a generalization of \cite[Theorem 4.1]{KMTLondon}.
We note that, when $\Delta=\Delta(A_1)$, it follows that
\begin{align}\label{classical1}
F(t,y;A_1)=\frac{te^{t\{y\}}}{e^t-1}
\end{align}
(see \cite[(3.14)]{KMTLondon}), the usual generating function of classical Bernoulli
functions $B_k(\cdot)$, that is
\begin{align}\label{classical2}
\frac{te^{t\{y\}}}{e^t-1}=\sum_{k=0}^{\infty}B_k(\{y\})\frac{t^k}{k!}.
\end{align}
Therefore we may regard that $P(\mathbf{k},\mathbf{y},\lambda;I;\Delta)$ is a
root-theoretic generalization of classical Bernoulli functions.

The fundamental philosophy of the present paper is to combine two expressions of
$S(\mathbf{s},\mathbf{y};I;\Delta)$, that is \eqref{sahen} and \eqref{eq:func_eq}.
It will produce the relation of the form
\begin{align}\label{phil}
&(\mbox{A signed sum of zeta-functions of root systems})\\
&\quad=(\mbox{A sum involving Bernoulli functions of root systems}).\notag
\end{align}
If the left-hand side of \eqref{phil} does not vanish identically, then this relation
gives a non-trivial functional relation among zeta-functions of root systems.
In Sections \ref{sec3} and \ref{sec4} we will show explicit
computations on the right-hand side of \eqref{phil}, and in Section \ref{sec7} 
we state several examples. 
On the other hand, concerning the left-hand side, it is important to know when
it does not vanish.    A criterion of the non-vanishing in terms of
Poincar{\'e} polynomials will be given in Sections \ref{sec5} and \ref{sec6}.

The proof of Theorems \ref{thm:main1} itself will be given in Section \ref{sec8}.
It is to be stressed that, though Theorem \ref{thm:main1} is a generalization of the previous
result in \cite{KMTLondon} where the case $I=\emptyset$ was treated, there are much more
technical difficulties when we consider the case of general $I$.    To overcome those
difficulties, the authors had to write a separate paper \cite{KMT2014}, whose results will be used
essentially in the course of the proof of Theorem \ref{thm:main1}.

Another main result in the present paper is the following converse theorem.
The generating function $F(\mathbf{t}_I,\mathbf{y},\lambda;I;\Delta)$ for general $I$ 
can be deduced,
in the following sense, from the case $I=\emptyset$.   That is, the generating function for 
$I=\emptyset$ knows `everything'.    This theorem will be proved in the last section.

\begin{theorem}\label{thm:main2}%[Theorema Egregium]
  Let $I\subset\{1,\ldots,r\}$. For $\lambda\in P_{I++}$, we have
  \begin{equation}\label{egregium}
    F(\mathbf{t}_I,\mathbf{y},\lambda;I;\Delta)
    =
    \Res_{\substack{t_\alpha=2\pi\sqrt{-1}\langle\alpha^\vee,\lambda\rangle\\\alpha\in\Delta_{I+}}}
    \Bigl(\prod_{\alpha\in\Delta_{I+}}\frac{1}{t_\alpha}\Bigr)
    F(\mathbf{t},\mathbf{y};\Delta),
  \end{equation}
where $\mathbf{t}_I=(t_\alpha)_{\alpha\in\Delta^*}$ and the meaning of $\Res$ is as
explained in Remark \ref{def_residue} below.
\end{theorem}

\begin{remark}\label{def_residue}
The function 
$\widetilde{F}(\mathbf{t},\mathbf{y};\Delta)=
(\prod_{\alpha\in\Delta_{I+}}1/t_{\alpha}) F(\mathbf{t},\mathbf{y};\Delta)$ is in
$|\Delta_+|$ complex variables.    
We give an order of the elements of the set $\Delta_{I+}$ by numbering as
$\alpha_1,\ldots,\alpha_n$, where $n=|\Delta_{I+}|$.
The definition of $\Res$ in the above theorem is given by the following iterated
procedure:
\begin{align}\label{iter_residue_F}
\Res_{\substack{t_\alpha=2\pi\sqrt{-1}\langle\alpha^\vee,\lambda\rangle\\\alpha\in\Delta_{I+}}}\widetilde{F}(\mathbf{t},\mathbf{y};\Delta)=
\Res_{t_{\alpha_n}=2\pi\sqrt{-1}\langle\alpha_n^\vee,\lambda\rangle}
  \cdots
\Res_{t_{\alpha_1}=2\pi\sqrt{-1}\langle\alpha_1^\vee,\lambda\rangle}
    \widetilde{F}(\mathbf{t},\mathbf{y};\Delta),
\end{align}
where on the right-hand side, each $\Res$ means the usual residue of a one-variable
function.
In the proof of the theorem we will see that {\it the value \eqref{iter_residue_F} does not
depend on the choice of the order of iteration}.
This is to be regarded as a part of the statement of Theorem \ref{thm:main2}, because
this does not hold in general for iterated residues.
The point $(2\pi\sqrt{-1}\langle\alpha^\vee,\lambda\rangle)_{\alpha\in\Delta_{I+}}$ 
is on some singular hyperplane of $ \widetilde{F}(\mathbf{t},\mathbf{y};\Delta)$,
but the limit process
$t_{\alpha}\to 2\pi\sqrt{-1}\langle\alpha^\vee,\lambda\rangle$ 
of calculating the residue is along some `generic' path in the sense that,
when the limit process on $t_{\alpha_j}$ is carried out, the remaining variables
$(t_{\alpha_{j+1}},\ldots,t_{\alpha_n})$ are located
outside any singular hyperplanes.
\end{remark}

%%%%%%%%%%%%%%%%%%%%%%%%%%%%%%%%%%%%%%%%%%%%%%%%%%%%%%%%%%%%%%%%%%%%%%%%%%%%%%%%%%%%%%%%%%
\section{Explicit functional relations ($\abs{I}=(r-1)$ case)}\label{sec3}
%%%%%%%%%%%%%%%%%%%%%%%%%%%%%%%%%%%%%%%%%%%%%%%%%%%%%%%%%%%%%%%%%%%%%%%%%%%%%%%%%%%%%%%%%%

%In this section, we consider recursive structure of zeta-functions.
In this and the next section we evaluate the generating functions more explicitly,
and deduce explicit functional relations.
We only discuss the cases $|I|=r-1$ and $|I|=1$, because the other cases are much more
complicated.    We use the notation $\delta_{\Box}$, where some condition is inserted
in $\Box$, defined as $\delta_{\Box}=1$ if the condition is satisfied, and $=0$ otherwise.

We first introduce the transposes $p^*_{\mathbf{V}_I^\perp}$ of the projections 
$p_{\mathbf{V}_I^\perp}$ (defined by \eqref{eq:proj}) by
\begin{equation}
  \begin{split}
    \langle u,p_{\mathbf{V}_I^\perp}(v)\rangle&=
    \langle u,v\rangle
    -\sum_{\beta\in\mathbf{V}_I}\langle u,\mu^{\mathbf{V}}_\beta\rangle\langle\beta^\vee,v\rangle
    \\
    &=
    \langle u-\sum_{\beta\in\mathbf{V}_I}\langle u,\mu^{\mathbf{V}}_\beta\rangle\beta^\vee,v\rangle
    \\
    &=
    \langle p^*_{\mathbf{V}_I^\perp}(u),v\rangle
  \end{split}
\end{equation}
for $v\in V$,
that is
\begin{equation}\label{transp_proj}
 p^*_{\mathbf{V}_I^\perp}(u)=u-\sum_{\beta\in\mathbf{V}_I}\langle u,\mu^{\mathbf{V}}_\beta\rangle\beta^\vee.
\end{equation}

Let
$I\subset \{1,\ldots,r\}$.
In this section we consider the situation  $|I^c|=1$, and put $I^c=\{k\}$.
Then $\Psi_I=\{\alpha_i\}_{i\in I}$,
and we see that
$$
\Delta^{*\vee}=\{\alpha^\vee=\sum_{i=1}^r a_i\alpha_i^\vee\in\Delta_+^\vee~|~a_k=\langle\alpha^\vee,\lambda_k\rangle\neq 0\}.
$$ 
Since $|\mathbf{V}_I|=1$ in the present case, we have
\begin{equation}
  \mathscr{V}_I=\{\mathbf{V}=\{\beta\}\cup\Psi_I\}_{\beta\in\Delta^*}.
\end{equation}
For $\mathbf{V}=\{\beta\}\cup\Psi_I\in\mathscr{V}_I$ and
$\gamma\in\Delta^*\setminus\{\beta\}$, 
from \eqref{transp_proj} we have
\begin{equation}\label{def-p-*}
 p^*_{\mathbf{V}_I^\perp}(\gamma^\vee)=\gamma^\vee-\langle\gamma^\vee,\mu^{\mathbf{V}}_\beta\rangle\beta^\vee.
\end{equation}
We put $b_i=b_i(\beta)=\langle\beta^\vee,\lambda_i\rangle$ ($1\leq i\leq r$) so that
\begin{equation}
  \beta^\vee=\sum_{i=1}^r b_i\alpha_i^\vee.
\end{equation}
%$\beta^\vee\in\Delta_*^\vee$ and 
%$\mathbf{V}^\vee=\{\beta^\vee\}\cup\{\alpha_j^\vee\}_{j\in I}$, we have
%$\mathbf{V}^\vee=\{\beta^\vee\}\cup\Psi_I^\vee$, we have
Then we find 
\begin{equation}\label{quotient}
%  \mu_\beta^{\mathbf{V}}=\frac{\lambda_i}{\langle\beta^\vee,\lambda_i\rangle}.  
  \mu_\beta^{\mathbf{V}}=\frac{\lambda_k}{b_k}.  
\end{equation}
Write $\mathbf{y}=y_1\alpha_1^\vee+\cdots+y_r\alpha_r^\vee$
and $\lambda=\sum_{\substack{i=1\\i\neq k}}^r m_i\lambda_i\in P_I$. 
Then
\begin{equation}\label{pppp}
  \begin{split}
    p_{\mathbf{V}_I^\perp}(\lambda)&=\lambda-\frac{\lambda_k}{b_k}\langle\beta^\vee,\lambda\rangle
    \\
    &=\sum_{\substack{i=1\\i\neq k}}^r m_i\lambda_i+\Bigl(-\sum_{\substack{i=1\\i\neq k}}^r\frac{b_i}{b_k}m_i\Bigr)\lambda_k
    \\
    &=\sum_{\substack{i=1\\i\neq k}}^r m_i \Bigl(\lambda_i-\frac{b_i}{b_k}\lambda_k\Bigr).
  \end{split}
\end{equation}
Note that
$Q^\vee/L(\mathbf{V}^\vee)=\{a_k\alpha_k^\vee\}_{0\leq a_k<b_k}$
and $|Q^\vee/L(\mathbf{V}^\vee)|=b_k$.
For $q=a_k \alpha_k^\vee\in Q^\vee$, from \eqref{pppp} we obtain
\begin{equation}
%  \begin{split}
    \langle \mathbf{y}+q,p_{\mathbf{V}_I^\perp}(\lambda)\rangle
    =
    \sum_{\substack{i=1\\i\neq k}}^r m_i \Bigl(y_i-\frac{b_i}{b_k}(y_k+a_k)\Bigr).
%    \sum_{\substack{j=1\\j\neq i}}^r m_jy_j+\Bigl(m_i-\sum_{j=1}^r\frac{b_j}{b_i}m_j\Bigr)(y_i+a_i).
  %   \\
  %   &=
  %   \sum_{i=2}^{j-1}m_i(y_i-y_1)+\sum_{i=j}^r m_iy_i.
  % \end{split}
\end{equation}
% \begin{equation}
%   \begin{split}
%     \langle p^*_{\mathbf{V}_I^\perp}(\mathbf{y}+q),\lambda\rangle
%     &=
%     y_1\langle e_j-e_2,\lambda\rangle+\sum_{i=2}^r y_i\langle\alpha_i^\vee,\lambda\rangle
%     \\
%     &=
%     \sum_{i=2}^{j-1}m_i(y_i-y_1)+\sum_{i=j}^r m_iy_i.
%   \end{split}
% \end{equation}
We fix $\phi$ such that $\langle\phi,\lambda_k\rangle>0$. Then for $q=a_k \alpha_k^\vee\in Q^\vee$,
\begin{equation}
  \{\mathbf{y}+q\}_{\mathbf{V},\beta}=\Bigl\{\frac{\langle\mathbf{y}+q,\lambda_k\rangle}{b_k}\Bigr\}=\Bigl\{\frac{y_k+a_k}{b_k}\Bigr\}.
\end{equation}
% By putting $t_{e_1-e_j}=t_j$ for $2\leq j\leq r+1$ 
Substituting the above results into \eqref{eq:exp_F}, 
we obtain the following form of the generating function
(under the identification $\mathbf{y}=(y_i)_{1\leq i\leq r}$ and
$\lambda=(m_i)_{1\leq i(\neq k)\leq r}$):
\begin{multline}\label{Fq-general}
    F((t_\beta)_{\beta\in\Delta^*},(y_i)_{1\leq i\leq r},(m_i)_{1\leq i(\neq k)\leq r};I;\Delta)  
\\
\begin{aligned}
  &=
  \sum_{\beta\in\Delta^*}
  \prod_{\gamma\in\Delta^*\setminus\{\beta\}}\frac{t_\gamma}{t_\gamma-\frac{\langle\gamma^\vee,\lambda_k\rangle}{b_k}t_\beta-2\pi\sqrt{-1}\langle p^*_{\mathbf{V}_I^\perp}(\gamma^\vee),\lambda\rangle}
  \\
  &\qquad\times
\frac{1}{b_k}\sum_{0\leq a_k<b_k}
  \exp\Bigl(2\pi\sqrt{-1}
  \sum_{\substack{i=1\\i\neq k}}^r m_i \Bigl(y_i-\frac{b_i}{b_k}(y_k+a_k)
  \Bigr)
  \Bigr)
  \frac{t_\beta\exp\Bigl(t_\beta\Bigl\{\dfrac{y_k+a_k}{b_k}\Bigr\}\Bigr)}{e^{t_\beta}-1},
\end{aligned}
\end{multline}
where $b_i=b_i(\beta)=\langle\beta^\vee,\lambda_i\rangle$.

For $\gamma^\vee\in\Delta^{*\vee}$, we see that
\begin{equation}
\label{eq:prj_rt}
%  \begin{split}
    p^*_{\mathbf{V}_I^\perp}(\gamma^\vee)
%    =\gamma^\vee-\frac{\langle\gamma^\vee,\lambda_i\rangle}{\langle\beta^\vee,\lambda_i\rangle}\beta^\vee
    =\gamma^\vee-\frac{\langle\gamma^\vee,\lambda_k\rangle}{b_k}\beta^\vee
    % \\
    % &
%=\sum_{\substack{j=1\\j\neq i}}^r c_j\alpha_j^\vee
%\in\frac{Q_I^\vee}{\langle\beta^\vee,\lambda_i\rangle}.
\in\frac{1}{b_k}Q_I^\vee,
%  \end{split}
\end{equation}
which may not be proportional to a coroot in $\Delta_I$. Therefore, 
observing \eqref{Fq-general} with \eqref{eq:prj_rt} we see that 
$S(\mathbf{s},\mathbf{y};I;\Delta)$ may not be necessarily written in terms of zeta-functions of root systems of lower ranks.
% \begin{equation}
% %  \label{eq:func_eq}
%     % \biggl(\prod_{\alpha\in \Delta^*}
%     % \frac{(2\pi\sqrt{-1})^{k_\alpha}}{k_\alpha!}\biggr)
%     \sum_{\lambda\in P_{I++}}
%     \biggr(
%     \prod_{\alpha\in\Delta_{I+}}
%     \frac{1}{\langle\alpha^\vee,\lambda\rangle^{s_\alpha}}
%     \biggl)
%     P(\mathbf{k},\mathbf{y},\lambda;I;\Delta).
% %  \end{split}
% \end{equation}
However, in the classical root systems of type $A,B,C$, special choices of $I$ give rise to recursive structures among zeta-functions of root systems.
To show this, 
in the following we will give explicit forms of 
\eqref{eq:prj_rt}.
Moreover we will present associated functional relations in the $A_r$ cases.
% $F(\mathbf{t}_I,\mathbf{y},\lambda;I;\Delta)$
% in special cases with $|I|=r-1$.

%%%%%%%%%%%%%%%%%%%%%%%%%%%%%%%%%%%%%%%%%%%%%%%%%%%%%%%%%%%%%%%%%%%%%%%%%%%%%%%

\subsection{$A_r$ Case}
We realize $\Delta_+^{\vee}(A_r)=\{e_i-e_j~|~1\leq i<j\leq r+1\}$,
where $e_j$ is the $j$th unit vector in $\mathbb{R}^{r+1}$. 
Then the variable $\mathbf{s}$ can be parametrized as
$\mathbf{s}=(s_{ij})_{1\leq i<j\leq r}$, and 
the zeta-function of type $A_r$ reads as
\begin{equation}
  \zeta_r((s_{ij})_{1\leq i<j\leq r},(y_i)_{1\leq i\leq r};A_r)
  =\sum_{m_1=1}^{\infty}\cdots
  \sum_{m_r=1}^{\infty}\frac{\exp(2\pi\sqrt{-1}\sum_{1\leq i\leq r}m_iy_i)}{\prod_{1\leq i<j\leq r+1}(m_i+\cdots+
    m_{j-1})^{s_{ij}}}.
\end{equation}
Choose $I=\{2,\ldots,r\}$ and $I^c=\{1\}$ as in the following diagram.

\begin{equation}
\setlength{\unitlength}{1pt}\begin{picture}(0,0)
\put(149,2){\oval(238,28)}
\end{picture}
\VN{\alpha_1}\E\VN{\alpha_2}\E\V\EO\V\E\V\E\V\E\VN{\alpha_{r}}
\end{equation}

Then 
\begin{equation}
  \Psi_I^{\vee}=\{\alpha_2^{\vee}=e_2-e_3,\ldots,\alpha_r^{\vee}=e_r-e_{r+1}\}
\end{equation}
and $\Delta^{*\vee}=\{e_1-e_j~|~2\leq j\leq r+1\}$. 
% Hence the dual basis of $\mathbf{V}^\vee=\{\beta^\vee=e_1-e_j=\alpha_1^\vee+\cdots+\alpha_{j-1}^\vee\}\cup\{\alpha_2^\vee,\ldots,\alpha_r^\vee\}$ is
% \begin{equation}
%   \{\lambda_1,\lambda_2-\lambda_1,\ldots,\lambda_{j-1}-\lambda_1,\lambda_j,\ldots,\lambda_r\}.  
% \end{equation}
Hence 
$$\mathscr{V}_I^{\vee}=
\{\mathbf{V}^{\vee}=\{e_1-e_j\}\cup\Psi_I^{\vee}\}_{2\leq j\leq r+1}.$$

Temporarily we fix one
$$
\beta^\vee=e_1-e_j=\alpha_1^\vee+\cdots+\alpha_{j-1}^\vee
$$ 
($2\leq j\leq r+1$).   Then we see that
\begin{equation}
  b_i=\langle\beta^\vee,\lambda_i\rangle=
  \begin{cases}
    1\qquad&(1\leq i<j),\\
    0\qquad&(j\leq i\leq r),
  \end{cases}
\end{equation}
and in particular $\mu_\beta^{\mathbf{V}}=\lambda_1/b_1=\lambda_1$ by \eqref{quotient}.
Since $\mathbf{V}_I^{\vee}=\{\beta^{\vee}\}$, 
for $\gamma^{\vee}=e_1-e_i\in\Delta^{*\vee}$ ($i\neq j$), from \eqref{def-p-*} we have
\begin{equation}
  \begin{split}
    p^*_{\mathbf{V}_I^\perp}(\gamma^\vee)
    &=e_1-e_i-\langle e_1-e_i,\lambda_1\rangle(e_1-e_j)
    \\
    &=e_1-e_i-\langle \alpha_1^\vee+\cdots+\alpha_{i-1}^\vee,\lambda_1\rangle(e_1-e_j)
    \\
    &=e_1-e_i-(e_1-e_j)
    \\
    &=e_j-e_i\in\Delta_I^\vee.
  \end{split}
\end{equation}
Since
$$
e_j-e_i=\left\{
\begin{array}{ll}
   -(\alpha_i^{\vee}+\cdots+\alpha_{j-1}^{\vee}) & (2\leq i\leq j-1),\\
   \alpha_j^{\vee}+\cdots+\alpha_{i-1}^{\vee} & (j+1\leq i\leq r+1),
\end{array}\right.
$$
we find that $\langle p^*_{\mathbf{V}_I^\perp}(\gamma^\vee), \lambda\rangle=-m_{ij}$
for $\lambda=m_2\lambda_2+\cdots+m_r\lambda_r$, where
\begin{align}
m_{ij}=\left\{
\begin{array}{ll}
   m_i+\cdots+m_{j-1} & (2\leq i\leq j-1),\\
   -(m_j+\cdots+m_{i-1})  &  (j+1\leq i\leq r+1).
\end{array}
\right.
\end{align}   

% Put $\mathbf{y}=y_1\alpha_1^\vee+\cdots+y_r\alpha_r^\vee$
% and $\lambda=m_2\lambda_2+\cdots+m_r\lambda_r\in P_I$. By noting $Q^\vee/L(\mathbf{V}^\vee)=\{0\}$, we obtain
% \begin{equation}
%   \begin{split}
%     \langle p^*_{\mathbf{V}_I^\perp}(\mathbf{y}),\lambda\rangle
%     &=
%     y_1\langle e_j-e_2,\lambda\rangle+\sum_{i=2}^r y_i\langle\alpha_i^\vee,\lambda\rangle
%     \\
%     &=
%     \sum_{i=2}^{j-1}m_i(y_i-y_1)+\sum_{i=j}^r m_iy_i.
%   \end{split}
% \end{equation}
% We fix $\phi$ such that $\langle\phi,\lambda_1\rangle>0$. Then
% \begin{equation}
%   \{\mathbf{y}\}_{\mathbf{V},\beta}=\{\langle\mathbf{y},\lambda_1\rangle\}=\{y_1\}.
% \end{equation}
 By putting $t_{e_1-e_i}=t_i$ for $2\leq i\leq r+1$, we can deduce from \eqref{Fq-general}
(with $k=1$) the following form of the generating function:
\begin{multline}\label{gen-funct-A}
    F((t_i)_{2\leq i\leq r+1},(y_i)_{1\leq i\leq r},(m_i)_{2\leq i\leq r};\{2,\ldots,r\};A_r)  
\\
\begin{aligned}
  &=
\sum_{j=2}^{r+1}
\prod_{\substack{2\leq i\leq r+1 \\ i\neq j}}\frac{t_i}{t_i-t_j+2\pi\sqrt{-1}m_{ij}}
%\\
%&\qquad\times
%\prod_{j<i\leq r+1}\frac{t_i}{t_i-t_j-2\pi\sqrt{-1}(m_j+\cdots+m_{i-1})}
\\
&\qquad\times
\exp\Bigl(2\pi\sqrt{-1}
\Bigl(
    \sum_{i=2}^{j-1}m_i(y_i-y_1)+\sum_{i=j}^r m_iy_i
%    \sum_{l=2}^r m_ly_l-\sum_{l=2}^{j-1}m_ly_1
\Bigr)
\Bigr)
\frac{t_j\exp(t_j\{y_1\})}{e^{t_j}-1}.
\end{aligned}
\end{multline}

Next we calculate the Taylor expansion of the right-hand side of the above formula.
First we note that
\begin{equation}
  \begin{split}
    \frac{t_i}{t_i-t_j+1/x_{ij}}
    &=  
    \frac{t_i x_{ij}}{1-(t_j-t_i)x_{ij}}\\
    &=  
    \sum_{n=0}^\infty x_{ij}^{n+1} t_i(t_j-t_i)^n 
    \\
    &=
    \sum_{n=0}^\infty x_{ij}^{n+1}
    \sum_{\substack{k,l\geq0\\k+l=n}}
    (-1)^{k}\binom{n}{l}t_j^lt_i^{k+1}
    \\
    &=
    \sum_{k\geq1,l\geq0}
    (-1)^{k-1}
    \binom{k+l-1}{l}x_{ij}^{k+l}t_j^lt_i^k
  \end{split}
\end{equation}
(with a certain indeterminate $x_{ij}$).
By use of this result and \eqref{classical2} we have
\begin{align*}
  &\Bigl(\prod_{\substack{2\leq i\leq r+1\\i\neq j}}\frac{t_i}{t_i-t_j+1/x_{ij}}\Bigr)\frac{t_j\exp(t_j\{y_1\})}{e^{t_j}-1}
  \\
  &=
  \sum_{\substack{k_i\geq1\ (i\neq j)\\l_i\geq 0\ (i\neq j)\\l_j\geq0}}
  \Bigl(
   \prod_{\substack{2\leq i\leq r+1\\i\neq j}}
  (-1)^{k_i-1}
  \binom{k_i+l_i-1}{l_i}x_{ij}^{k_i+l_i}t_j^{l_i}t_i^{k_i}\Bigr)
  B_{l_j}(\{y_1\})\frac{t_j^{l_j}}{l_j!}
  \\
  &=
  \sum_{\substack{k_i\geq1\ (i\neq j)\\k_j\geq0}}
  \Bigl(
  \sum_{\substack{l_2,\ldots,l_{r+1}\geq0\\l_2+\cdots+l_{r+1}=k_j}}
  \frac{B_{l_j}(\{y_1\})}{l_j!}
  \prod_{\substack{2\leq i\leq r+1\\i\neq j}}
  (-1)^{k_i-1}
  \binom{k_i+l_i-1}{l_i}x_{ij}^{k_i+l_i}\Bigr)
  \\
  &\qquad\times
  \prod_{2\leq i\leq r+1}t_i^{k_i}.
\end{align*}
Applying this to the right-hand side of \eqref{gen-funct-A}, we obtain the following 
conclusion.
\begin{theorem}\label{Th-Fq-Ar}
In the case $\Delta=\Delta(A_r)$ and $I=\{2,\ldots,r\}$, we have
  \begin{multline}
    F((t_i)_{2\leq i\leq r+1},(y_j)_{1\leq j\leq r},(m_i)_{2\leq i\leq r};\{2,\ldots,r\};A_r)  
\\
\begin{aligned}
  &=
\sum_{j=2}^{r+1}
\prod_{\substack{2\leq i\leq r+1 \\ i\neq j}}\frac{t_i}{t_i-t_j+2\pi\sqrt{-1}m_{ij}}
%\sum_{j=2}^{r+1}
%\prod_{2\leq i<j}\frac{t_i}{t_i-t_j+2\pi\sqrt{-1}(m_i+\cdots+m_{j-1})}
%\\
%&\qquad\times
%\prod_{j<i\leq r+1}\frac{t_i}{t_i-t_j-2\pi\sqrt{-1}(m_j+\cdots+m_{i-1})}
\\
&\qquad\times
\exp\Bigl(2\pi\sqrt{-1}
\Bigl(
    \sum_{i=2}^{j-1}m_i(y_i-y_1)+\sum_{i=j}^r m_iy_i
\Bigr)
\Bigr)
\frac{t_j\exp(t_j\{y_1\})}{e^{t_j}-1}
\\
&=\sum_{k_2,\ldots,k_{r+1}\geq 0}P((k_i)_{2\leq i\leq r+1},(y_i)_{1\leq i\leq r},(m_i)_{2\leq i\leq r};\{2,\ldots,r\};A_r)\frac{t_2^{k_2}\cdots t_{r+1}^{k_{r+1}}}{k_2!\cdots k_{r+1}!},
\end{aligned}
\end{multline}
where
\begin{multline}\label{P-explicit}
P((k_i)_{2\leq i\leq r+1},(y_i)_{1\leq i\leq r},(m_i)_{2\leq i\leq r};\{2,\ldots,r\};A_r)
\\
\begin{aligned}
  &=k_2!\cdots k_{r+1}!
  \sum_{j=2}^{r+1}  \Bigl(\prod_{\substack{i=2\\i\neq j}}^{r+1} \delta_{k_i\neq 0} \Bigr)
  \exp\Bigl(2\pi\sqrt{-1}
  \Bigl(
  \sum_{i=2}^{j-1}m_i(y_i-y_1)+\sum_{i=j}^r m_iy_i
  \Bigr)
  \Bigr)
  \\
  &\times
  \Bigl(
  \sum_{\substack{l_2,\ldots,l_{r+1}\geq0\\l_2+\cdots+l_{r+1}=k_j}}
\frac{B_{l_j}(\{y_1\})}{l_j!}
  \prod_{\substack{2\leq i\leq r+1\\i\neq j}}
    (-1)^{k_i-1}
    \binom{k_i+l_i-1}{l_i}\Bigl(\frac{1}{2\pi\sqrt{-1}m_{ij}}\Bigr)^{k_i+l_i}\Bigr).
\end{aligned}
\end{multline}
%with
%\begin{equation}
%  m_{ij}=
%  \begin{cases}
%    m_i+\cdots+m_{j-1}\qquad&(2\leq i<j)\\
%    -(m_j+\cdots+m_{i-1})\qquad&(j<i\leq r+1).
%  \end{cases}
%\end{equation}
\end{theorem}

As an application of this theorem, we will give explicit functional relations.
We compute \eqref{sahen} and \eqref{eq:func_eq} in the case of type $A_r$,
with $s_{1j}=k_{1j}\in\mathbb{N}$ for $2\leq j\leq r+1$.

First, it is well known that $W(A_r)\simeq \mathfrak{S}_{r+1}$, the $(r+1)$-th symmetric 
group, as the permutation group on $\{e_1,\ldots,e_{r+1}\}$.
Via this identification 
we see that
\begin{align}
  W^I=&\{\id,\sigma_1,\sigma_1\sigma_2,\ldots,\sigma_1\sigma_2\cdots \sigma_r\}\\
  =&\{\id,(1\;2),(1\;2\;3),\ldots,(1\;2\cdots  r+1)\}\notag
\end{align}
where $\sigma_j=(j\; j+1)$.
In fact, we can easily check that 
$$w=\sigma_1\sigma_2\cdots\sigma_j=(1\;2\cdots j+1)$$ 
for $0\leq j\leq r$ satisfies $w^{-1}\Delta_{I+}^\vee\subset\Delta_+^\vee$, and those elements exhausts $W^I$
because $|W^I|=|W|/|W_I|=(r+1)!/r!=r+1$ (see \eqref{W-decomp} below).

Next, for $2\leq i\leq r$ and $1\leq j\leq r$, we have
\begin{align}
  \sigma_j\cdots \sigma_1\alpha_1^\vee&=e_{j+1}-e_1=-\alpha_1^\vee-\cdots-\alpha_j^\vee,\\
  \sigma_j\cdots \sigma_1\alpha_i^\vee&=
                              \begin{cases}
                                \alpha_{i-1}^\vee\qquad&(i\leq j),\\
                                \alpha_{i-1}^\vee+\alpha_i^\vee\qquad&(i=j+1),\\
                                \alpha_i^\vee\qquad&(i\geq j+2).
                              \end{cases}
\end{align}
Therefore, for $w=\sigma_1\cdots \sigma_j\in W^I$, we have
\begin{align*}
&w^{-1}\mathbf{y}=\sigma_j\cdots \sigma_1(y_1\alpha_1^{\vee}+\cdots+y_r\alpha_r^{\vee})\\
&=-y_1(\alpha_1^\vee+\cdots+\alpha_j^\vee)
+\sum_{i=2}^j y_i\alpha_{i-1}^{\vee}+y_{j+1}(\alpha_j^{\vee}+\alpha_{j+1}^{\vee})
+\sum_{i=j+2}^r y_i\alpha_i^{\vee}\\
&=\sum_{i=1}^j(y_{i+1}-y_1)\alpha_i^\vee+
\sum_{i=j+1}^r y_i\alpha_i^\vee.
\end{align*}

Lastly for $w=\sigma_1\cdots \sigma_j\in W^I$, we can see that
\begin{equation}
  \Delta_{w^{-1}}=\{\alpha_1^{\vee}+\cdots+\alpha_i^{\vee}=e_1-e_{i+1}~|~1\leq i\leq j\}.
\end{equation}
Therefore from \eqref{sahen} we obtain
\begin{align}\label{sahen-lin-comb}
  &S(\mathbf{s},\mathbf{y};\{2,\ldots,r\};A_r)=
  \sum_{j=0}^r\Bigl(\prod_{i=1}^j(-1)^{k_{1,i+1}}\Bigr)\\
&\qquad\times\zeta_r((s_{p_j q_j})_{1\leq p<q\leq r+1},(y_2-y_1,\ldots,y_{j+1}-y_1,y_{j+1},\ldots,y_r);A_r),
\notag
\end{align}
where $p_j=(1\;2\cdots j+1)p$ and $q_j=(1\;2\cdots j+1)q$.

%(-1)^r\prod_{j=2}^{r+1}\frac{(2\pi\sqrt{-1})^{k_{1,j}}}{k_{1,j}!}
%(-1)^r\prod_{j=2}^{r+1}(2\pi\sqrt{-1})^{k_{1,j}}
On the other hand, \eqref{eq:func_eq} in the present case reads as
\begin{align*}
&S(\mathbf{s},\mathbf{y};\{2,\ldots,r\};A_r)
=(-1)^r\left(\prod_{j=2}^{r+1}\frac{(2\pi\sqrt{-1})^{k_{1j}}}{k_{1j}!}\right)\\
&\times
\sum_{m_2,\ldots,m_r=1}^{\infty}\left(\prod_{2\leq p<q\leq r+1}\frac{1}
{\langle e_p-e_q,m_2\lambda_2+\cdots+m_r\lambda_r\rangle^{s_{pq}}}\right)\\
&\times
P(\mathbf{k},\mathbf{y},(m_i)_{2\leq i\leq r};\{2,\ldots,r\};A_r)
\end{align*}
with $\mathbf{k}=(k_{1j})_{2\leq j\leq r+1}$.
Substituting \eqref{P-explicit} (with $k_i=k_{1i}$) into the right-hand side, 
we obtain
\begin{align}\label{uhen-lin-comb}
&S(\mathbf{s},\mathbf{y};\{2,\ldots,r\};A_r)\\
&=(-1)^r(2\pi\sqrt{-1})^{\sum_{2\leq i\leq r+1}k_{1i}}
\sum_{j=2}^{r+1}
  \sum_{\substack{l_2,\ldots,l_{r+1}\geq0\\l_2+\cdots+l_{r+1}=k_{1j}}}
\frac{B_{l_j}(\{y_1\})}{l_j!}
\notag\\
&\qquad\times
\Bigl(  \prod_{\substack{2\leq i\leq r+1\\i\neq j}}
    (-1)^{k_{1i}-1}
    \binom{k_{1i}+l_i-1}{l_i}\Bigl(\frac{(-1)^{\delta_{i>j}}}{2\pi\sqrt{-1}}\Bigr)^{k_{1i}+l_i}\Bigr)
\notag\\
&\qquad\times \sum_{m_2,\ldots,m_r=1}^{\infty}
 \exp\Bigl(2\pi\sqrt{-1}
  \Bigl(
  \sum_{i=2}^{j-1}m_i(y_i-y_1)+\sum_{i=j}^r m_iy_i
  \Bigr)
  \Bigr)\notag\\
  &\qquad\times\prod_{\substack{2\leq i\leq r+1\\i\neq j}}
  \frac{1}{|m_{ij}|^{k_{1i}+l_i}}\prod_{2\leq p<q\leq r+1}\frac{1}
  {(m_p+\cdots+m_{q-1})^{s_{pq}}}
\notag\\
&=
%(-1)^r\prod_{j=2}^{r+1}\frac{(2\pi\sqrt{-1})^{k_{1,j}}}{k_{1,j}!}
%(-1)^r\prod_{j=2}^{r+1}(2\pi\sqrt{-1})^{k_{1,j}}
(-1)^r(2\pi\sqrt{-1})^{\sum_{2\leq i\leq r+1}k_{1i}}
\sum_{j=2}^{r+1}
  \sum_{\substack{l_2,\ldots,l_{r+1}\geq0\\l_2+\cdots+l_{r+1}=k_{1j}}}
\frac{(-1)^{k_{1,2}+\cdots+k_{1,j-1}+l_{j+1}+\cdots+l_{r+1}-r+1}}{(2\pi\sqrt{-1})^{l_j+\sum_{2\leq i\leq r+1}k_{1i}}}
\notag\\
&\qquad\times
\frac{B_{l_j}(\{y_1\})}{l_j!}
\Bigl(  \prod_{\substack{2\leq i\leq r+1\\i\neq j}}
    \binom{k_{1i}+l_i-1}{l_i}\Bigr)
\notag\\
&\qquad\times
\zeta_{r-1}((s_{pq}+\delta_{p<j}\delta_{q=j}(k_{1p}+l_p)+\delta_{p=j}\delta_{q>j}(k_{1q}+l_q))_{2\leq p<q\leq r+1},
\notag\\
&\qquad\qquad\qquad
(y_2-y_1,\ldots,y_{j-1}-y_1,y_j,\ldots,y_r);A_{r-1}),
\notag
%\end{aligned}
\end{align}
where in the last equality we have used
\begin{equation}
  \sum_{2\leq i\leq r+1,i\neq j}(k_{1i}+l_i)=-l_j+\sum_{2\leq i\leq r+1}k_{1i}.
\end{equation}
Comparing \eqref{sahen-lin-comb} and \eqref{uhen-lin-comb}, we now arrive at the 
following explicit form of functional relations.
\begin{theorem}\label{Th-Fq-2}
For $(s_{ij})_{1\leq i<j\leq r+1}$ with $s_{1j}=k_{1j}$ ($2\leq j\leq r+1$), we have
\begin{align}\label{Fq-rel-1}
&\sum_{j=0}^r\Bigl(\prod_{i=1}^j(-1)^{k_{1,i+1}}\Bigr)\\
% &\qquad\times\zeta_r((s_{(1\cdots j\,j+1)pq})_{1\leq p<q\leq r+1},(y_2-y_1,\ldots,y_{j+1}-y_1,y_{j+1},\ldots,y_r);A_r)
  &\qquad\times\zeta_r((s_{p_jq_j})_{1\leq p<q\leq r+1},(y_2-y_1,\ldots,y_{j+1}-y_1,y_{j+1},\ldots,y_r);A_r)
\notag\\
% \begin{aligned}
&=
%\prod_{j=2}^{r+1}\frac{1}{k_{1,j}!}
-\sum_{j=2}^{r+1}
  \sum_{\substack{l_2,\ldots,l_{r+1}\geq0\\l_2+\cdots+l_{r+1}=k_{1j}}}
{(-1)^{k_{1,2}+\cdots+k_{1,j-1}+l_{j+1}+\cdots+l_{r+1}}}{(2\pi\sqrt{-1})^{l_j}}
\notag\\
&\qquad\times
\frac{B_{l_j}(\{y_1\})}{l_j!}
\Bigl(  \prod_{\substack{2\leq i\leq r+1\\i\neq j}}
    \binom{k_{1i}+l_i-1}{l_i}\Bigr)
\notag\\
&\qquad\times
\zeta_{r-1}((s_{pq}+\delta_{p<j}\delta_{q=j}(k_{1p}+l_p)+\delta_{p=j}\delta_{q>j}(k_{1q}+l_q))_{2\leq p<q\leq r+1},
\notag\\
&\qquad\qquad\qquad
(y_2-y_1,\ldots,y_{j-1}-y_1,y_j,\ldots,y_r);A_{r-1}),
\notag
%\end{aligned}
\end{align}
where $p_j=(1\;2\cdots j+1)p$ and $q_j=(1\;2\cdots j+1)q$.
\end{theorem}

The cases $r=2,3$ of this theorem will be written down as Examples \ref{ex-1}, \ref{ex-2}
in Section \ref{sec7}.

%%%%%%%%%%%%%%%%%%%%%%%%%%%%%%%%%%%%%%%%%%%%%%%%%%%%%%%%%%%%%%%%%%%%%%%%%%%%%%%%%%

\subsection{$C_r$ Case}
We realize 
$$\Delta_+^{\vee}=\{e_i\pm e_j~|~1\leq i<j\leq r\}\cup\{e_j~|~1\leq j\leq r\}.$$ 
Then $\Psi^{\vee}=\{\alpha_1^{\vee}=e_1-e_2,\ldots,\alpha_{r-1}^{\vee}=e_{r-1}-e_r,\alpha_r^{\vee}=e_r\}$, and the fundamental weights are given by
$\lambda_i=e_1+\cdots+e_i$ for $1\leq i\leq r$.
Choose $I=\{2,\ldots,r\}$ and $I^c=\{1\}$ as in the following diagram.

\begin{equation}
\setlength{\unitlength}{1pt}\begin{picture}(0,0)
\put(149,2){\oval(238,28)}
%\put(44,2){\circle{26}}%
%\put(254,2){\circle{26}}%
\end{picture}
\VN{\alpha_1}\E\VN{\alpha_2}\E\V\EO\V\E\V\E\V\EB{\E}{\E}\VN{\!\!\!\!\!\!\alpha_{r}}
\end{equation}

Then 
\begin{equation}
  \Psi_I^{\vee}=\{\alpha_2^{\vee}=e_2-e_3,\ldots,\alpha_{r-1}^{\vee}=e_{r-1}-e_r,\alpha_r^{\vee}=e_r\}
\end{equation}
and $\Delta^{*\vee}=\{e_1\pm e_j~|~2\leq j\leq r\}\cup\{e_1\}$. 
Hence 
\begin{equation}
  \mathscr{V}_I^{\vee}=\{\{e_1-e_j\}\cup\Psi_I^{\vee}\}_{2\leq j\leq r}
\cup\{\{e_1+e_j\}\cup\Psi_I^{\vee}\}_{2\leq j\leq r}
\cup\{\{e_1\}\cup\Psi_I^{\vee}\}.
\end{equation}
For $\beta^\vee=e_1-e_j$ ($2\leq j\leq r$), we see that
\begin{equation}
  b_l=\langle\beta^\vee,\lambda_l\rangle=
  \begin{cases}
    1\qquad&(1\leq l<j),\\
    0\qquad&(j\leq l\leq r),
  \end{cases}
\end{equation}
and in particular $\mu_\beta^{\mathbf{V}}=\lambda_1=e_1$.
Therefore for $\gamma^{\vee}=e_1-e_i\in\Delta^{*\vee}$ ($i\neq j$)
and $\gamma^{\vee}=e_1+e_i\in\Delta^{*\vee}$,
\begin{equation}
  \begin{split}
    p^*_{\mathbf{V}_I^\perp}(\gamma^\vee)
    &=e_1\pm e_i-\langle e_1\pm e_i,\lambda_1\rangle(e_1-e_j)
    \\
    &=e_1\pm e_i-(e_1-e_j)
    \\
    &=
    \begin{cases}    
    e_j\pm e_i\in\Delta_I^\vee & (i\neq j),\\
    2e_j\in 2\Delta_I^{\vee} & (i=j),
    \end{cases}
  \end{split}
\end{equation}
and 
\begin{equation}
  \begin{split}
    p^*_{\mathbf{V}_I^\perp}(e_1)
    &=e_1-\langle e_1,\lambda_1\rangle(e_1-e_j)
    \\
    &=e_1-(e_1-e_j)
    \\
    &=e_j\in\Delta_I^\vee.
  \end{split}
\end{equation}
Next for $\beta^\vee=e_1+e_j$ ($2\leq j\leq r$), we see that
\begin{equation}
  b_l=\langle\beta^\vee,\lambda_l\rangle=
  \begin{cases}
    1\qquad&(1\leq l<j),\\
    2\qquad&(j\leq l\leq r),
  \end{cases}
\end{equation}
and in particular $\mu_\beta^{\mathbf{V}}=\lambda_1=e_1$.
Therefore for $\gamma^{\vee}=e_1+e_i\in\Delta^{*\vee}$ ($i\neq j$)
and $\gamma^{\vee}=e_1-e_i\in\Delta^{*\vee}$,
\begin{equation}
  \begin{split}
    p^*_{\mathbf{V}_I^\perp}(\gamma^\vee)
    &=e_1\pm e_i-\langle e_1\pm e_i,\lambda_1\rangle(e_1+e_j)
    \\
    &=e_1\pm e_i-(e_1+e_j)
    \\
    &=
    \begin{cases}
      -e_j\pm e_i\in\Delta_I^\vee\qquad&(i\neq j), \\
      -2e_j\in 2\Delta_I^\vee\qquad&(i=j),
    \end{cases}
  \end{split}
\end{equation}
and 
\begin{equation}
  \begin{split}
    p^*_{\mathbf{V}_I^\perp}(e_1)
    &=e_1-\langle e_1,\lambda_1\rangle(e_1+e_j)
    \\
    &=e_1-(e_1+e_j)
    \\
    &=-e_j\in\Delta_I^\vee.
  \end{split}
\end{equation}
Lastly for $\beta^\vee=e_1$, we see that
\begin{equation}
  b_l=\langle\beta^\vee,\lambda_l\rangle=
  1\qquad(1\leq l\leq r),
\end{equation}
and in particular $\mu_\beta^{\mathbf{V}}=\lambda_1=e_1$.
Therefore for $\gamma^{\vee}=e_1\pm e_i\in\Delta^{*\vee}$,
\begin{equation}
  \begin{split}
    p^*_{\mathbf{V}_I^\perp}(\gamma^\vee)
    &=e_1\pm e_i-\langle e_1\pm e_i,\lambda_1\rangle e_1
    \\
    &=e_1\pm e_i-e_1
    \\
    &=\pm e_i\in\Delta_I^\vee.
  \end{split}
\end{equation}

Now calculate $\langle p^*_{\mathbf{V}_I^\perp}(\gamma^\vee), \lambda\rangle$.
Consider the case when $\beta^{\vee}=e_1-e_j$ ($2\leq j\leq r$).
For $\gamma^{\vee}=e_1-e_i$ ($i\neq j$), we have
$$
\langle p^*_{\mathbf{V}_I^\perp}(\gamma^\vee), \lambda\rangle
=\langle e_j-e_i,m_2\lambda_2+\cdots+m_r\lambda_r\rangle=-m_{ij},
$$
where the second equality is obtained by using
$\lambda_i=e_1+\cdots+e_i$, or by using (the second member of) the facts
\begin{align}\label{koredemoii}
\begin{cases}
e_i=\sum_{i\leq k\leq r}\alpha_k^{\vee} \quad (1\leq i\leq r),\\
e_i-e_j=\sum_{i\leq k<j}\alpha_k^{\vee} \quad (1\leq i<j\leq r),\\
e_i+e_j=\sum_{i\leq k<j}\alpha_k^{\vee}+2\sum_{j\leq k\leq r}\alpha_k^{\vee} \quad 
(1\leq i<j\leq r).
\end{cases}
\end{align}
Using again $\lambda_i=e_1+\cdots+e_i$ or \eqref{koredemoii}, 
for $\gamma^{\vee}=e_1+e_i$ we have
$$
\langle p^*_{\mathbf{V}_I^\perp}(\gamma^\vee), \lambda\rangle
=\langle e_j+e_i,m_2\lambda_2+\cdots+m_r\lambda_r\rangle=m_{ijr},
$$
where
\begin{align}
m_{ijr}=
\begin{cases}
  m_i+\cdots+m_{j-1}+2(m_j+\cdots+m_r) & (2\leq i<j),\\
  2(m_j+\cdots+m_r) & (i=j),\\
  m_j+\cdots+m_{i-1}+2(m_i+\cdots+m_r) & (j<i\leq r).
\end{cases}
\end{align}
For $\gamma^{\vee}=e_1$, 
$$
\langle p^*_{\mathbf{V}_I^\perp}(\gamma^\vee), \lambda\rangle
=\langle e_j,m_2\lambda_2+\cdots+m_r\lambda_r\rangle=m_j+\cdots+m_r=m_{j,r+1}.
$$
Similarly we evaluate $\langle p^*_{\mathbf{V}_I^\perp}(\gamma^\vee), \lambda\rangle$
for other $\beta^{\vee}$.

We now insert those data into \eqref{Fq-general}.
Putting $t_{e_1\pm e_i}=t_{\pm i}$ for $2\leq i\leq r$ and $t_{2e_1}=t_1$,
we obtain the following form of the generating function:

\begin{theorem}
In the case $\Delta=\Delta(C_r)$ and $I=\{2,\ldots,r\}$ we have
\begin{align}\label{3-41}
    &F(t_1,(t_{\pm i})_{2\leq i\leq r},(y_j)_{1\leq j\leq r},(m_i)_{2\leq i\leq r};\{2,\ldots,r\};C_r)  
\\
%\begin{aligned}
&=
\sum_{j=2}^{r}
\prod_{\substack{2\leq i\leq r \\ i\neq j}}\frac{t_{-i}}{t_{-i}-t_{-j}+2\pi\sqrt{-1}m_{ij}}
\prod_{2\leq i\leq r}\frac{t_{+i}}{t_{+i}-t_{-j}-2\pi\sqrt{-1}m_{ijr}}
\notag\\
%  &=
%\sum_{j=2}^{r}
%\prod_{2\leq i<j}\frac{t_{-i}}{t_{-i}-t_{-j}+2\pi\sqrt{-1}(m_i+\cdots+m_{j-1})}
%\notag\\
%&\qquad\times
%\prod_{j<i\leq r}\frac{t_{-i}}{t_{-i}-t_{-j}-2\pi\sqrt{-1}(m_j+\cdots+m_{i-1})}
%\notag\\
%&\qquad\times
%\prod_{2\leq i\leq r}\frac{t_{+i}}{t_{+i}-t_{-j}-2\pi\sqrt{-1}m_{ijr}}
%\notag\\
%&\qquad\times
%\prod_{2\leq i\leq j}\frac{t_{+i}}{t_{+i}-t_{-j}-2\pi\sqrt{-1}(m_i+\cdots+m_{j-1}+2(m_j+\cdots+m_{r}))}
%\notag\\
%&\qquad\times
%\prod_{j<i\leq r}\frac{t_{+i}}{t_{+i}-t_{-j}-2\pi\sqrt{-1}(m_j+\cdots+m_{i-1}+2(m_i+\cdots+m_{r}))}
%\notag\\
&\qquad\times\frac{t_1}{t_1-t_{-j}-2\pi\sqrt{-1}m_{j,r+1}}
\notag\\
%&\qquad\times\frac{t_1}{t_1-t_{-j}-2\pi\sqrt{-1}(m_j+\cdots+m_{r})}
%\notag\\
&\qquad\times
\exp\Bigl(2\pi\sqrt{-1}
\Bigl(
    \sum_{i=2}^{j-1}m_i(y_i-y_1)+\sum_{i=j}^r m_iy_i
\Bigr)
\Bigr)
\frac{t_{-j}\exp(t_{-j}\{y_1\})}{e^{t_{-j}}-1}
\notag\\
&+
\sum_{j=2}^{r}
\prod_{2\leq i\leq j}\frac{t_{-i}}{t_{-i}-t_{+j}+2\pi\sqrt{-1}m_{ijr}}
\prod_{\substack{2\leq i\leq r \\ i\neq j}}\frac{t_{+i}}{t_{+i}-t_{+j}-2\pi\sqrt{-1}m_{ij}}
\notag\\
%\sum_{j=2}^{r}
%\prod_{2\leq i\leq j}\frac{t_{-i}}{t_{-i}-t_{+j}+2\pi\sqrt{-1}(m_i+\cdots+m_{j-1}+2(m_j+\cdots+m_{r}))}
%\notag\\
%&\qquad\times
%\prod_{j<i\leq r}\frac{t_{-i}}{t_{-i}-t_{+j}+2\pi\sqrt{-1}(m_j+\cdots+m_{i-1}+2(m_i+\cdots+m_{r}))}
%\notag\\
%&\qquad\times
%\prod_{2\leq i<j}\frac{t_{+i}}{t_{+i}-t_{+j}-2\pi\sqrt{-1}(m_i+\cdots+m_{j-1})}
%\notag\\
%&\qquad\times
%\prod_{j<i\leq r}\frac{t_{+i}}{t_{+i}-t_{+j}+2\pi\sqrt{-1}(m_j+\cdots+m_{i-1})}
%\notag\\
&\qquad\times\frac{t_1}{t_1-t_{+j}+2\pi\sqrt{-1}m_{j,r+1}}
\notag\\
%&\qquad\times\frac{t_1}{t_1-t_{+j}+2\pi\sqrt{-1}(m_j+\cdots+m_{r})}
%\notag\\
&\qquad\times
\exp\Bigl(2\pi\sqrt{-1}
\Bigl(
    \sum_{i=2}^{j-1}m_i(y_i-y_1)+\sum_{i=j}^{r} m_i(y_i-2y_1) 
\Bigr)
\Bigr)
\frac{t_{+j}\exp(t_{+j}\{y_1\})}{e^{t_{+j}}-1}
\notag\\
&+
\prod_{2\leq i\leq r}\frac{t_{-i}}{t_{-i}-t_{1}+2\pi\sqrt{-1}m_{i,r+1}}
\prod_{2\leq i\leq r}\frac{t_{+i}}{t_{+i}-t_{1}-2\pi\sqrt{-1}m_{i,r+1}}
\notag\\
%\prod_{2\leq i\leq r}\frac{t_{-i}}{t_{-i}-t_{1}+2\pi\sqrt{-1}(m_i+\cdots+m_r)}
%\notag\\
%&\qquad\times
%\prod_{2\leq i\leq r}\frac{t_{+i}}{t_{+i}-t_{1}-2\pi\sqrt{-1}(m_i+\cdots+m_r)}
%\notag\\
&\qquad\times
\exp\Bigl(2\pi\sqrt{-1}
\Bigl(
    \sum_{i=2}^{r}m_i(y_i-y_1)
\Bigr)
\Bigr)
\frac{t_{1}\exp(t_{1}\{y_1\})}{e^{t_{1}}-1}.
\notag
%\end{aligned}
\end{align}
\end{theorem}

From this theorem, it is possible to deduce the explicit forms of Bernoulli functions and
functional relations for the root system of type $C_r$.
The case $r=3$ will be discussed more explicitly in Section \ref{sec7}.

%%%%%%%%%%%%%%%%%%%%%%%%%%%%%%%%%%%%%%%%%%%%%%%%%%%%%%%%%%%%%%%%%%%%%%%%%%%%%%%%%%%%%%%%
\section{Explicit functional relations ($\abs{I}=1$ case)}\label{sec4}
%%%%%%%%%%%%%%%%%%%%%%%%%%%%%%%%%%%%%%%%%%%%%%%%%%%%%%%%%%%%%%%%%%%%%%%%%%%%%%%%%%%%%%%%%
In this section we consider the case $I=\{i\}$ and $\mathbf{y}=\mathbf{0}$.
In this case the result can be stated in terms of the Lerch zeta-function
\begin{equation}
  \phi(s,u)=\sum_{n=1}^\infty \frac{e^{2\pi\sqrt{-1}u n}}{n^s}.
\end{equation}

\begin{theorem}
Let $s_\alpha=k_\alpha\in\mathbb{Z}_{\geq2}$ for $\alpha\in \Delta_+\setminus\{\alpha_i\}$ and
$s_{\alpha_i}\in\mathbb{C}$. Let
$|\mathbf{k}|=\sum_{\alpha\in\Delta_+\setminus\{\alpha_i\}}k_\alpha$.
Let $X_i=\{\nu=\{\langle q,\mu^{\mathbf{V}}_{\alpha_i}\rangle\}~|~
\mathbf{V}\in\mathscr{V}_I,q\in Q^\vee/L(\mathbf{V}^\vee)\}\subset\mathbb{Q}$.
Then we have
\begin{multline}
  \sum_{w\in W^I}
  \Bigl(\prod_{\alpha\in\Delta_{w^{-1}}}(-1)^{-k_{\alpha}}\Bigr)
  \zeta_r(w^{-1}\mathbf{s},0;\Delta)
  \\
  =(-1)^{\abs{\Delta_+}-1}
  \biggl(\prod_{\alpha\in \Delta_+\setminus\{\alpha_i\}}
  \frac{(2\pi\sqrt{-1})^{k_\alpha}}{k_\alpha!}\biggr)
%  \sum_{0\leq\nu<N_i}
  \sum_{\nu\in X_i}
  \sum_{j=0}^{|\mathbf{k}|}
  \frac{b_{\mathbf{k}\nu j}}{(2\pi\sqrt{-1})^j}
  \phi(s_{\alpha_i}+j,\nu),
\end{multline}
where
$b_{\mathbf{k}\nu j}\in\mathbb{Q}$ is given by
\begin{multline}
  \label{eq:coef_a}
  \sum_{\mathbf{k}\in \mathbb{N}_0^{\abs{\Delta^*}}}
  \sum_{\nu\in X_i}
  \sum_{j=0}^{|\mathbf{k}|}
  b_{\mathbf{k}\nu j}x^jy^\nu
  \prod_{\alpha\in \Delta^*}
  \frac{t_\alpha^{k_\alpha}}{k_\alpha!}\\
  =
  \sum_{\mathbf{V}\in\mathscr{V}_I}
  \prod_{\gamma\in \Delta^*\setminus\mathbf{V}_I}
  \frac{t_\gamma}
  {t_\gamma-\sum_{\beta\in\mathbf{V}_I}
    t_\beta\langle\gamma^\vee,\mu^{\mathbf{V}}_\beta\rangle
    -\langle\gamma^\vee,\mu^{\mathbf{V}}_{\alpha_i}\rangle/x}
  \\
  \times
  \frac{1}{\abs{Q^\vee/L(\mathbf{V}^\vee)}}
    \sum_{q\in Q^\vee/L(\mathbf{V}^\vee)}
    y^{\{\langle q,\mu^{\mathbf{V}}_{\alpha_i}\rangle\}} %\exp(2\pi\sqrt{-1}n\nu)
  \prod_{\beta\in\mathbf{V}_I}
  \frac{t_\beta\exp
    (t_\beta
    \{q\}_{\mathbf{V},\beta})}{e^{t_\beta}-1},
\end{multline}
where $x$ and $y$ are indeterminates.
\end{theorem}
\begin{proof}
Since $I=\{i\}$, we have $\Delta_{I+}=\{\alpha_i\}$ and
$\Delta^*=\Delta_+\setminus\{\alpha_i\}$.
  Note that
for $\lambda=n\lambda_i\in P_{I++}$,
we have $p_{\mathbf{V}_I^\perp}(\lambda)=n\mu^{\mathbf{V}}_{\alpha_i}$
by \eqref{eq:proj},
and hence for $u\in V$,
\begin{equation}
\langle u,p_{\mathbf{V}_I^\perp}(\lambda)\rangle
=n\langle u,\mu^{\mathbf{V}}_{\alpha_i}\rangle.
\end{equation}
Therefore
\begin{multline}\label{F-Taylor}
  F(\mathbf{t},0,\lambda;I;\Delta)
  =
    \sum_{\mathbf{V}\in\mathscr{V}_I}
    \prod_{\gamma\in \Delta^*\setminus\mathbf{V}_I}
    \frac{t_\gamma}
    {t_\gamma-\sum_{\beta\in\mathbf{V}_I}
      t_\beta\langle\gamma^\vee,\mu^{\mathbf{V}}_\beta\rangle
      -2\pi\sqrt{-1}n\langle\gamma^\vee,\mu^{\mathbf{V}}_{\alpha_i}\rangle}
    \\
    \times
    \frac{1}{\abs{Q^\vee/L(\mathbf{V}^\vee)}}
    \sum_{q\in Q^\vee/L(\mathbf{V}^\vee)}
    \exp(2\pi\sqrt{-1}n\{\langle q,\mu^{\mathbf{V}}_{\alpha_i}\rangle\})
    \prod_{\beta\in\mathbf{V}_I}
    \frac{t_\beta\exp
      (t_\beta
      \{q\}_{\mathbf{V},\beta})}{e^{t_\beta}-1}.
\end{multline}

% \if0
% Consider the Taylor expansion of the right-hand side with respect to
% $\mathbf{t}$.
% Non-rational coefficients in the expansion come only
% from the factors $2\pi\sqrt{-1}n$ (in the denominator) and
% $\exp(2\pi\sqrt{-1}n \{\langle q,\mu^{\mathbf{V}}_{\alpha_i}\rangle\})$.
% \fi
Consider the Taylor expansion of the right-hand side with respect to
$\mathbf{t}$.
Putting $x=(2\pi\sqrt{-1}n)^{-1}$ and $y=\exp(2\pi\sqrt{-1}n)$, 
the right-hand side of \eqref{F-Taylor} is equal to the right-hand side of
\eqref{eq:coef_a}.
We write the Taylor expansion as the left-hand side of \eqref{eq:coef_a} with rational coefficients $b_{\mathbf{k}\nu j}$;
here we remark that in \eqref{eq:coef_a}, 
the highest degree of $x$ is at most $|\mathbf{k}|$, because when
$\langle\gamma^\vee,\mu^{\mathbf{V}}_{\alpha_i}\rangle\neq 0$, 
\begin{align}
&  \frac{t_\gamma}
  {t_\gamma-\sum_{\beta\in\mathbf{V}_I}
    t_\beta\langle\gamma^\vee,\mu^{\mathbf{V}}_\beta\rangle
    -\langle\gamma^\vee,\mu^{\mathbf{V}}_{\alpha_i}\rangle/x}\\
&  \quad =
  \frac{x t_\gamma}
  {x t_\gamma -\sum_{\beta\in\mathbf{V}_I}
     x t_\beta \langle\gamma^\vee,\mu^{\mathbf{V}}_\beta\rangle
    -\langle\gamma^\vee,\mu^{\mathbf{V}}_{\alpha_i}\rangle}\notag
\end{align}
and hence $x$ in the expansion appears necessarily together with $t_\alpha$.

From \eqref{eq:coef_a} we have
\begin{equation}
  P(\mathbf{k},0,\lambda;I;\Delta)=
  \sum_{\nu\in X_i}\sum_{j=0}^{|\mathbf{k}|}
  \frac{b_{\mathbf{k}\nu j}}{(2\pi\sqrt{-1}n)^j}
  \exp(2\pi\sqrt{-1}n\nu).
\end{equation}
Therefore from Theorem \ref{thm:main1} we obtain
\begin{equation}
  \begin{split}
   & S(\mathbf{s},0;I;\Delta)
    =(-1)^{\abs{\Delta_+}-1}
    \biggl(\prod_{\alpha\in \Delta_+\setminus\{\alpha_i\}}
    \frac{(2\pi\sqrt{-1})^{k_\alpha}}{k_\alpha!}\biggr)
    \sum_{n=1}^\infty
    \frac{1}{n^{s_{\alpha_i}}}
    P(\mathbf{k},0,\lambda;I;\Delta)
    \\
    &\quad=(-1)^{\abs{\Delta_+}-1}
    \biggl(\prod_{\alpha\in \Delta_+\setminus\{\alpha_i\}}
    \frac{(2\pi\sqrt{-1})^{k_\alpha}}{k_\alpha!}\biggr)
    \sum_{\nu\in X_i}
    \sum_{j=0}^{|\mathbf{k}|}
    \frac{b_{\mathbf{k}\nu j}}{(2\pi\sqrt{-1})^j}
    \phi(s_{\alpha_i}+j,\nu).
  \end{split}
\end{equation}
Combining this and \eqref{sahen} we obtain the assertion of the theorem.
\end{proof}

%%%%%%%%%%%%%%%%%%%%%%%%%%%%%%%%%%%%%%%%%%%%%%%%%%%%%%%%%%%%%%%%%%%%%%%%%%%%%%%%%%%%
\section{Poincar\'e polynomials}\label{sec5}
%%%%%%%%%%%%%%%%%%%%%%%%%%%%%%%%%%%%%%%%%%%%%%%%%%%%%%%%%%%%%%%%%%%%%%%%%%%%%%%%%%%%%

Now we turn our attention to the left-hand side of \eqref{phil}.
We begin with the relation
  \begin{equation}\label{W-decomp}
    W_IW^I=W,
  \end{equation}
which follows by the definitions of $W_I, W^I$ or \cite[Lemma 1]{KMTWitten3}.
Therefore $w\in W$ is uniquely written as $w=xy$ with $x\in W_I$, $y\in W^I$.
We note
\begin{align}\label{w-x-y-decomp}
\Delta_{w^{-1}}=x\Delta_{y^{-1}}\sqcup \Delta_{x^{-1}}.
\end{align}
This is well-known, but we supply a proof.
Let $x=\sigma_{\nu_1}\cdots \sigma_{\nu_j}$, $y=\sigma_{\nu_{j+1}}\cdots \sigma_{\nu_{k}}$ be reduced expressions 
by simple
reflections.    Then $w=\sigma_{\nu_{1}}\cdots \sigma_{\nu_{k}}$ is also a reduced expression because of the
length identity $l(w)=l(x)+l(y)$ (\cite[Section 1.10, Proposition (c)]{Hum}).
Let $\beta_{ik}=\sigma_{\nu_{k}} \sigma_{\nu_{k-1}}\cdots \sigma_{\nu_{i+1}}(\alpha_{\nu_i})$ ($1\leq i\leq k-1$)
and $\beta_{kk}=\alpha_k$.   Then $\Delta_w=\{\beta_{1k},\ldots,\beta_{kk}\}$ 
(\cite[Section 1.7, p.14]{Hum}).    Similarly we have
$\Delta_x=\{\beta_{1j},\ldots,\beta_{jj}\}$ and
$\Delta_y=\{\beta_{j+1,k},\ldots,\beta_{kk}\}$.    Therefore
$\Delta_{w}=y^{-1}\Delta_{x}\sqcup \Delta_{y}$.
Applying this argument to $w^{-1}=y^{-1}x^{-1}$, we obtain \eqref{w-x-y-decomp}.

Let $\mathbf{u}=(u_\alpha)_{\alpha\in\Delta_+}$ be a vector of indeterminates.
For $X\subset\Aut(\Delta)$, we define a generalization of Poincar\'e polynomials due to Macdonald
\cite{Mac} as
\begin{equation}
  X(\mathbf{u})=\sum_{w\in X}\prod_{\alpha\in\Delta_{w^{-1}}}u_\alpha.
\end{equation}
When $X=W^I$ and $\mathbf{u}=((-1)^{-s_{\alpha}})_{\alpha\in\Delta_+}$, the
corresponding Poincar{\'e} polynomial is
\begin{align}\label{sum-of-coeff}
W^I(\mathbf{u})=\sum_{w\in W^I}\prod_{\alpha\in\Delta_{w^{-1}}}
(-1)^{-s_{\alpha}},
\end{align}
which is the sum of coefficients of \eqref{sahen}.
If this is not zero, then the left-hand side of \eqref{phil} does not vanish
identically.

\begin{lemma}[{cf.~\cite[Section 1.11, p.\,21]{Hum}}]
\label{lm:WIWI}
Assume $u_\alpha=u_\beta$ if $\lVert\alpha\rVert=\lVert\beta\rVert$. Then
  \begin{equation}\label{WIWI=W}
  W_I(\mathbf{u})W^I(\mathbf{u})=W(\mathbf{u}).
\end{equation}
\end{lemma}
\begin{proof}
The assumption for $u_\alpha$ implies $u_\alpha=u_{x^{-1}\alpha}$.
Therefore, using \eqref{W-decomp} and \eqref{w-x-y-decomp} we have
  \begin{equation}
    \begin{split}
      &W(\mathbf{u})
      =\sum_{w\in W}\prod_{\alpha\in\Delta_{w^{-1}}}u_\alpha
      =\sum_{x\in W_I}\sum_{y\in W^I}
      \Bigl(\prod_{\alpha\in x\Delta_{y^{-1}}}u_\alpha\Bigr)
      \Bigl(\prod_{\alpha\in \Delta_{x^{-1}}}u_\alpha\Bigr)\\
      &=\sum_{x\in W_I}\sum_{y\in W^I}
      \Bigl(\prod_{\alpha\in\Delta_{y^{-1}}}u_\alpha\Bigr)
      \Bigl(\prod_{\alpha\in\Delta_{x^{-1}}}u_\alpha\Bigr)
      =W_I(\mathbf{u})W^I(\mathbf{u}).
    \end{split}
  \end{equation}
\end{proof}

Assume $u_\alpha=u$ for all $\alpha\in\Delta_+$, that is, $\mathbf{u}=(u,\ldots,u)$,
which we write $(u)$ for brevity.  
By Lemma \ref{lm:WIWI}, we have
\begin{equation}\label{W-decomp2}
W^I((u))=\frac{W((u))}{W_I((u))}.
\end{equation}
In this case the numerator and the denominator of the right-hand side are classical
Poincar{\'e} polynomials, which have the following product expression (due to
Chevalley):
\begin{align}
\label{eq:Wu}
  W((u))&=\prod_{i=1}^r\frac{u^{d_i}-1}{u-1}=\prod_{i=1}^r(1+\cdots+u^{e_i}),\\
\label{eq:WIu}
  W_I((u))&=\prod_{i\in I}\frac{u^{d'_i}-1}{u-1}=\prod_{i\in I}(1+\cdots+u^{e'_i}),
\end{align}
where $d_i$ and $d'_i$ (resp.~$e_i$ and $e'_i$) are degrees (resp.~exponents) of the Weyl groups $W$ and $W_I$.
In fact, the first equalities are \cite[Section 3.15, Theorem]{Hum}, and then the
second equalities are immediate in view of \cite[Section 3.19, Theorem]{Hum}.
The following is the table of degrees (\cite[Section 3.7]{Hum}).
Here we note that $W(B_r)=W(C_r)$ and hence the their degrees agree.
\bigskip

\begin{center}
\begin{tabular}{c|c}
  Type & $\{d_1,\ldots,d_r\}=\{e_1+1,\ldots,e_r+1\}$\\
\hline
  $A_r$&$2,3,\ldots,r+1$\\
  $B_r$&$2,4,\ldots,2r$\\
  $C_r$&$2,4,\ldots,2r$\\
  $D_r$&$2,4,\ldots,2r-2,r$\\
  $E_6$&$2,5,6,8,9,12$\\
  $E_7$&$2,6,8,10,12,14,18$\\
  $E_8$&$2,8,12,14,18,20,24,30$\\
  $F_4$&$2,6,8,12$\\
  $G_2$&$2,6$\\
  \hline
\end{tabular}
\end{center}
\bigskip

Assume that $\Delta$ is a non-simply laced root system
and
$\Delta=\Delta_1\cup\Delta_2$, where each $\Delta_i$ consists of all roots of the same length.
Then each $\Delta_i$ is a root system.
Let $\Delta_{1+}=\Delta_+\cap \Delta_1$
and $\Delta_{1-}=\Delta_-\cap \Delta_1$.
Let $J$ be the set of indices determined by 
$\Psi_2=\Psi\cap\Delta_2=\{\alpha_i\}_{i\in J}$. 
\begin{prop}\label{1-2}
$W(\Delta)((u),(1))=\abs{W_J}W(\Delta_1)((u))$,
where the left-hand side means that $u_{\alpha}=u$ for $\alpha\in\Delta_1$ and
$u_{\alpha}=1$ for $\alpha\in\Delta_2$. 
\end{prop}
\begin{proof}
We show $W(\Delta_1)\cap W_J=\{\id\}$.
Let $w\in W(\Delta_1)\cap W_J$.
Then $w=\sigma_{i_1}\cdots \sigma_{i_l}$ with $i_k\in J$.
Since $\sigma_{i_k}\Delta_+=(\Delta_+\setminus\{\alpha_{i_k}\})\cup\{-\alpha_{i_k}\}$,
we have
$\sigma_{i_k}\Delta_{-}=(\Delta_-\setminus\{-\alpha_{i_k}\})\cup\{\alpha_{i_k}\}
=\Delta_{-}$ and hence $\sigma_{i_k}\Delta_{1-}=\Delta_{1-}$ (because 
$\pm\alpha_{i_k}\notin\Delta_1$).    This implies that
 $\Delta_{1+}\cap w^{-1}\Delta_{1-}=\emptyset$.
Therefore the length of $w$ in $W(\Delta_1)$ is $0$ and
hence $w=\id$.

We define $f:W(\Delta_1)\to W^J$ by $f(w)=w^J$ using the decomposition $w=w^Jw_J$
 for $w\in W(\Delta_1)$.
We show that $f$ is bijective.
For
$w,\Tilde{w}\in W(\Delta_1)$,
assume $f(w)=f(\Tilde{w})$. Then $ww_J^{-1}=\Tilde{w}\Tilde{w}_J^{-1}$ and
$w^{-1}\Tilde{w}=w_J^{-1}\Tilde{w}_J\in W(\Delta_1)\cap W_J=\{\id\}$ by the previous paragraph,
which implies $w=\Tilde{w}$ and the injectivity of $f$.
Fix $x\in W^J$.
Let $x=\sigma_{i_1}\cdots \sigma_{i_m}$ be a reduced expression.
We decompose
$(i_1,\ldots,i_m)$ into two subsequences
as
$(j_1,\ldots,j_q)$ and $(k_1,\ldots,k_{m-q})$ 
such that
 $j_l\in J$ and $k_l\not\in J$.
Let $y=\sigma_{j_q}\cdots \sigma_{j_1}\in W_J$ and consider $w=xy$. Then 
by use of $\sigma_\alpha \sigma_\beta=\sigma_\beta \sigma_{\sigma_\beta \alpha}$, 
we carry each $\sigma_{j_l}$ forward in $y$ until it cancels the same element at the original position in $x$.
Hence we see that $w$ can be written as $\sigma_{\beta_1}\cdots \sigma_{\beta_{m-q}}$
with $\beta_p=\sigma_{j_1}\cdots \sigma_{j_h}\alpha_{k_p}\in \Delta_1$ for some $h$.
Thus we have $w=xy\in W(\Delta_1)$.
By the uniqueness of the decomposition, we see that $f$ is surjective and hence bijective.

In the first paragraph of the proof we have seen that
$w_J\Delta_{1-}=\Delta_{1-}$ for $w_J\in W_J$.
Noting this fact and the bijectivity of $f$, we obtain
\begin{equation}
  \label{eq:w1}
  \begin{split}
    W(\Delta_1)((u))
    &=\sum_{w\in W(\Delta_1)}\prod_{\alpha\in(\Delta_1)_{w^{-1}}}u\\
    &=\sum_{w\in W(\Delta_1)}\prod_{\alpha\in\Delta_{1+}\cap w^J w_J\Delta_{1-}}u\\
    &=\sum_{w\in W(\Delta_1)}\prod_{\alpha\in\Delta_{1+}\cap w^J\Delta_{1-}}u\\
    &=\sum_{x\in W^J}\prod_{\alpha\in\Delta_{1+}\cap x\Delta_{1-}}u.% \\
    % &=\sum_{x\in W^J}\Bigl(\prod_{\alpha\in\Delta_{1+}\cap x\Delta_{1-}}u\Bigr)
    % \Bigl(\prod_{\alpha\in\Delta_{2+}\cap x\Delta_{2-}}1\Bigr)\\
    % &=W^J(u,1).
  \end{split}
\end{equation}
Furthermore
we have
\begin{equation}
  \label{eq:w2}
  \begin{split}
    W(\Delta)((u),(1))
    &=\sum_{w\in W(\Delta)}\Bigl(\prod_{\alpha\in\Delta_{w^{-1}}\cap\Delta_1}u\Bigr)
\Bigl(\prod_{\alpha\in\Delta_{w^{-1}}\cap\Delta_2}1\Bigr)
\\
    &=\sum_{w\in W(\Delta)}\prod_{\alpha\in\Delta_{w^{-1}}\cap\Delta_1}u\\
    &=\sum_{y\in W_J}\sum_{x\in W^J}\prod_{\alpha\in\Delta_{1+}\cap xy\Delta_{1-}}u\\
    &=\sum_{y\in W_J}\sum_{x\in W^J}\prod_{\alpha\in\Delta_{1+}\cap x\Delta_{1-}}u.
% \\
%     &=\abs{W_J}W^J(u,1).
  \end{split}
\end{equation}
Combining \eqref{eq:w1} and \eqref{eq:w2}, we obtain the result.
\end{proof}

We denote $W(X_r)=W(\Delta)$ for a root system of type $X_r$.
For a root system of type $X_r=B_r,C_r,F_4,G_2$, let $\Delta_1=\Delta_L(X_r)$ be 
the set of all long roots,
$\Delta_2=\Delta_S(X_r)$ that of all short roots.
\begin{corollary}
  \label{cor:twolen}
  \begin{align*}
    W(B_r)((u),(1))&=W(C_r)((1),(u))=\abs{W(A_1)}W(D_r)((u))=2W(D_r)((u)), \\
    W(C_r)((u),(1))&=W(B_r)((1),(u))=\abs{W(A_{r-1})}W(A_1^r)((u))=r!(W(A_1)((u)))^r, \\
    W(F_4)((u),(1))&=W(F_4)((1),(u))=\abs{W(A_2)}W(D_4)((u))=6W(D_4)((u)), \\
    W(G_2)((u),(1))&=W(G_2)((1),(u))=\abs{W(A_1)}W(A_2)((u))=2W(A_2)((u)).  
  \end{align*}
\end{corollary}

\begin{proof}
These are simple consequences of Proposition \ref{1-2}. 
We just notice that 

(i) $\Delta_L(B_r)\simeq\Delta_S(C_r)$ and
$\Delta_S(B_r)\simeq\Delta_L(C_r)$,

(ii) $\Delta_L(B_r)\simeq\Delta(D_r)$, $\Delta_L(C_r)\simeq\Delta(A_1^r)$, 
$\Delta_L(F_4)\simeq\Delta(D_4)$, $\Delta_L(G_2)\simeq\Delta(A_2)$, and
$\Psi_2(B_r)\simeq\Psi(A_1)$, $\Psi_2(C_r)\simeq\Psi(A_{r-1})$,
$\Psi_2(F_4)\simeq\Psi(A_2)$, $\Psi_2(G_2)\simeq\Psi(A_1)$ 
(which can be seen from the list of roots of each system; see 
\cite[Planche]{Bou}, \cite[Section 2.14]{Sam}),

(iii) $W(A_r)\simeq\mathfrak{S}_{r+1}$, and hence
$|W(A_r)|=(r+1)!$.
\end{proof}

%%%%%%%%%%%%%%%%%%%%%%%%%%%%%%%%%%%%%%%%%%%%%%%%%%%%%%%%%%%%%%%%%%%%%%%%%%%%%%%%%%%%%%%%%%
\section{The non-vanishing of the sum of coefficients}\label{sec6}
%%%%%%%%%%%%%%%%%%%%%%%%%%%%%%%%%%%%%%%%%%%%%%%%%%%%%%%%%%%%%%%%%%%%%%%%%%%%%%%%%%%%%%%%%%%

Now we carry out the case study when \eqref{sum-of-coeff} vanishes, or does not vanish.
Assume 
$s_\alpha=s_{w\alpha}$ for all $w\in W^I$ and $\alpha\in\Delta$.
Assume 
$s_\alpha=k_\alpha\in\mathbb{N}$ for all $\alpha\in\Delta\setminus \Delta_I$,
and let $u_\alpha=(-1)^{-k_\alpha}$ for those $\alpha$.

\begin{lemma}
Let $y\in W^I$. Then
  $\Delta_{y^{-1}}\cap \Delta_{I+}=\emptyset$.
\end{lemma}
\begin{proof}
Assume $\alpha\in\Delta_{y^{-1}}\cap \Delta_{I+}$. Then
$s_\alpha\in W_I$. Let $w=s_\alpha y$; this expression gives the unique decomposition
of $w$.
Then $\Delta_{w^{-1}}=s_\alpha\Delta_{y^{-1}}\sqcup \Delta_{s_\alpha}$ by 
\eqref{w-x-y-decomp}.
By the definition $\Delta_{w^{-1}}\subset\Delta_+$, 
while
${-\alpha}=s_{\alpha}\alpha\in s_\alpha\Delta_{y^{-1}}$, which leads to the contradiction.
\end{proof}

From the above lemma, we see that the product on the right-hand side of
\begin{equation}
  W^I(\mathbf{u})=\sum_{w\in W^I}\prod_{\alpha\in\Delta_{w^{-1}}}u_\alpha
\end{equation}
consists of only $\alpha\in\Delta_+\setminus\Delta_{I+}$.

Now consider several cases.    Assume $I\subsetneq\{1,\ldots,r\}$.

\subsection{Case 1}
This is the trivial case.
If all $k_\alpha$ are even, then $u_\alpha=1$ 
for all $\alpha\in\Delta$,
and we have $W^I((1))=\abs{W^I}>0$.

\subsection{Case 2}
If all $k_\alpha$ are odd, then we have to evaluate
\eqref{W-decomp2}
at $\mathbf{u}=(-1)$.
Let $K$ and $K_I$ be the sets of indices of even degrees (i.e., odd exponents) of the
Weyl groups $W$ and $W_I$ respectively, given as
\begin{align}
K&=\{i~|~1\leq i\leq r,d_i\in2\mathbb{Z}\}=\{i~|~1\leq i\leq r,e_i\in2\mathbb{Z}+1\},
\\
K_I&=\{i~|~i\in I,d'_i\in2\mathbb{Z}\}=\{i~|~i\in I,e'_i\in2\mathbb{Z}+1\}.
\end{align}
From \eqref{eq:Wu} and \eqref{eq:WIu} we see that $W((-1))=0$ if $K\neq \emptyset$, 
and $W_I((-1))=0$ if $K_I\neq\emptyset$.
Comparing the orders of zeros of the both sides of \eqref{WIWI=W}, 
we find that 
\begin{equation}
\abs{K}\geq
\abs{K_I},
\end{equation}
hence $W^I((-1))$ does not vanish if and only if 
\begin{equation}\label{K=KI}
  \abs{K}=\abs{K_I}.
\end{equation}
If it holds, then by \eqref{eq:Wu}, \eqref{eq:WIu} and l'H\^opital's rule
\begin{equation}\label{lHopital}
  \begin{split}
    W^I((-1))&=
    \lim_{u\to-1}W^I((u))\\
    &=
    \lim_{u\to-1}
    \frac{\prod_{i\in K^c}(1+\cdots+u^{e_i})}
    {\prod_{i\in K_I^c}(1+\cdots+u^{e'_i})}
    \frac{\prod_{i\in K}(u^{d_i}-1)}
    {\prod_{i\in K_I}(u^{d'_i}-1)}
    \\
    &=
    \frac{\prod_{i\in K^c}1}{\prod_{i\in K_I^c}1}
    \frac{\prod_{i\in K}d_i}{\prod_{i\in K_I}d'_i}
    \\
    &=
    \frac{\prod_{i\in K}d_i}{\prod_{i\in K_I}d'_i}
    \in\mathbb{N}.
  \end{split}
\end{equation}

The above argument especially implies the following
\begin{claim}
If all degrees of a root system $X_r$ are even, then 
$W^I((-1))$ vanishes for any choice of $I\subsetneq\{1,\ldots,r\}$.
\end{claim}

This is obvious from \eqref{K=KI}, because for any $I\subsetneq\{1,\ldots,r\}$, we have
$|K|>|K_I|$.
From this claim and the table in Section \ref{sec5} we find that $W^I((-1))$
always vanishes for $B_r, C_r, D_{2k}, E_7, E_8, F_4$, and $G_2$.
Therefore the only root systems for which \eqref{lHopital} can be applied are
$A_r, D_{2k+1}$, and $E_6$.
%The following is a table of several of such cases. %where their values survive.
%for which the above \eqref{K=KI}, \eqref{lHopital} are valid.
\begin{theorem}
  $W^I((-1))$ does not vanish only in one of the following cases.

  \bigskip
  \begin{center}
    \def\arraystretch{2.2}
  \begin{tabular}{c|c|c}
  Type of $\Delta$ & Type of $\Delta_I$ & $W^I((-1))$ \\
  \hline
  $A_r$ & $A_{r_1}\times\cdots\times A_{r_n}$ &
$\dfrac{[(r+1)/2]!}{[(r_1+1)/2]!\cdots[(r_n+1)/2]!}$ \\
  $D_{2k+1}$ & $D_{2k}$ & $\dfrac{2\cdot 4\cdots (4k-2)\cdot 4k}{2\cdot 4\cdots (4k-2)\cdot 2k}=2$ \\
  $E_6$ & $D_5$ & $\dfrac{2\cdot 6\cdot 8\cdot 12}{2\cdot 4\cdot 6\cdot 8}=3$ \\
  $E_6$ & $D_4$ & $\dfrac{2\cdot 6\cdot 8\cdot 12}{2\cdot 4\cdot 6\cdot 4}=6$ \\
  \hline
\end{tabular}
\end{center}
\bigskip
where $[(r_1+1)/2]+\cdots+[(r_n+1)/2]=[(r+1)/2]$.
\end{theorem}
\begin{proof}
  Assume that $\Delta$ is of type $E_6$. Then $\Delta_I$ must be a subroot system such that there are $4$ even degrees in it, which implies that $\Delta_I$ is either of type $D_4$ or $D_5$. 

  Assume that $\Delta$ is of type $D_{2k+1}$. Then $\Delta_I$ must be a subroot system such that there are $2k$ even degrees in it, which implies that $\Delta_I$ is of type $D_{2k}$. 

  Assume that $\Delta$ is of type $A_{r}$.
  By noting that there are $[(r+1)/2]$ even degrees and they are $2,4,\ldots,2[(r+1)/2]$, and that their product is
  \begin{equation}
    2\cdot 4\cdots\cdot (2[(r+1)/2])=2^{[(r+1)/2]}[(r+1)/2]!,
  \end{equation}
  we obtain the result.
\end{proof}

For example, see the following diagrams for the pairs $(A_{2k},A_{2k-1})$ and 
$(A_4,A_1\times A_2)$, where the set of enclosed nodes is $\Psi_I$.

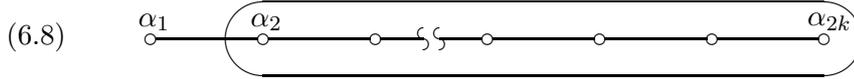
\begin{gather}
\setlength{\unitlength}{1pt}\begin{picture}(0,0)
\put(149,2){\oval(238,28)}
%\put(44,2){\circle{26}}%
%\put(254,2){\circle{26}}%
\end{picture}
\VN{\alpha_1}\E\VN{\alpha_2}\E\V\EO\V\E\V\E\V\E\VN{\alpha_{2k}}
\\[8mm]
\setlength{\unitlength}{1pt}
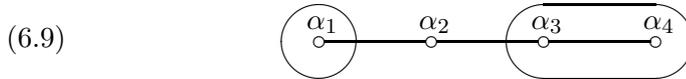
\begin{picture}(0,0)\put(2,2){\circle{26}}%
\put(107,2){\oval(70,28)}
%\put(107,2){\circle{4}}
%\put(86,2){\circle{26}}%
%\put(128,2){\circle{26}}%
\end{picture}
  \VN{\alpha_1}\E\VN{\alpha_2}\E\VN{\alpha_3}\E\VN{\alpha_4}
\end{gather}

\bigskip
Here are some examples of the case when $\Delta$ is of type $A$.
\begin{center}
\begin{tabular}{c|c|c}
  Type of $\Delta$ & Type of $\Delta_I$ & $W((-1))$ %Coefficient
  \\
\hline
$A_{2k}$ & $A_{2k-1}$ & $2\cdot 4\cdots 2k/2\cdot 4\cdots 2k=1$ \\
$A_3$ & $A_1^2$ & $2\cdot 4/2\cdot 2=2$ \\
$A_4$ & $A_1^2$ & $2\cdot 4/2\cdot 2=2$ \\
$A_4$ & $A_1\times A_2$ & $2\cdot 4/2\cdot 2=2$ \\
$A_5$ & $A_1^3$ & $2\cdot 4\cdot 6/2\cdot 2\cdot 2=6$ \\
% $D_{2k+1}$ & $D_{2k}$ & $2\cdot 4\cdots 4k/2\cdot 4\cdots (4k-2)\cdot 2k=2$ \\
% $E_6$ & $D_5$ & $2\cdot 6\cdot 8\cdot 12/2\cdot 4\cdot 6\cdot 8=3$ \\
% $E_6$ & $D_4$ & $2\cdot 6\cdot 8\cdot 12/2\cdot 4\cdot 6\cdot 4=6$ \\
\hline
\end{tabular}
\end{center}

\subsection{Case 3}\label{sec-6-3}
Assume $k_\alpha$ are odd for $\alpha\in\Delta_1$
and $k_\beta$ are even for $\beta\in\Delta_2$. Let $\mathbf{u}=((u),(1))$.
Then we have
\begin{equation}
  W^I(\mathbf{u})=\frac{W((u),(1))}{W_I((u),(1))}
=\frac{\abs{W_J}W(\Delta_1)((u))}{\abs{W_{I\cap J}}W(\Delta_1\cap\Delta_I)((u))}
\end{equation}
by Proposition \ref{1-2}.
  By using an argument similar to Case 2, we can calculate $W^I((-1),(1))$
as follows.
Let $K_1$ and $K_{1I}$ be the sets of indices of even degrees of the
Weyl groups $W(\Delta_1)$ and $W(\Delta_1\cap\Delta_I)$ respectively, given as
\begin{align}
K_1&=\{i~|~\alpha_i\in\Delta_1,d_i\in2\mathbb{Z}\},
\\
K_{1I}&=\{i~|~\alpha_i\in \Delta_1\cap \Delta_I,d'_i\in2\mathbb{Z}\}.
\end{align}
% From \eqref{eq:Wu} and \eqref{eq:WIu} we see that $W((-1))=0$ if $K\neq \emptyset$, 
% and $W_I((-1))=0$ if $K_I\neq\emptyset$.
% Comparing the orders of zeros of the both sides of \eqref{WIWI=W}, 
% we find that 
% \begin{equation}
% \abs{K}\geq
% \abs{K_I},
% \end{equation}
Then we see that $W^I((-1),(1))$ does not vanish if and only if 
\begin{equation}%\label{K=KI}
  \abs{K_1}=\abs{K_{1I}}.
\end{equation}
If it holds, then %by \eqref{eq:Wu}, \eqref{eq:WIu} and l'H\^opital's rule
\begin{equation}%\label{lHopital}
  \begin{split}
    W^I((-1),(1))&=
    \lim_{u\to-1}W^I((u),(1))\\
    &=
    \frac{\abs{W_J}}{\abs{W_{I\cap J}}}
    \lim_{u\to-1}
    \frac{\prod_{i\in K_1^c}(1+\cdots+u^{e_i})}
    {\prod_{i\in K_{1I}^c}(1+\cdots+u^{e'_i})}
    \frac{\prod_{i\in K_1}(u^{d_i}-1)}
    {\prod_{i\in K_{1I}}(u^{d'_i}-1)}
    \\
    &=
    \frac{\abs{W_J}}{\abs{W_{I\cap J}}}
    \frac{\prod_{i\in K_1^c}1}{\prod_{i\in K_{1I}^c}1}
    \frac{\prod_{i\in K_1}d_i}{\prod_{i\in K_{1I}}d'_i}
    \\
    &=
    \frac{\abs{W_J}}{\abs{W_{I\cap J}}}
    \frac{\prod_{i\in K_1}d_i}{\prod_{i\in K_{1I}}d'_i}
    \in\mathbb{N},
  \end{split}
\end{equation}
where the complement of sets of indices means that in $\{i~|~\alpha_i\in\Delta_1\}$.

\begin{theorem}
  $W^I((-1),(1))$ does not vanish only in one of the following cases.

% \end{theorem}

%   The following is a table of several cases where their values survive.
\bigskip

\begin{center}
    \def\arraystretch{2.2}
\begin{tabular}{c|c|c|c}
  Type of $\Delta$ & Type of $\Delta_I$ & $\Delta_1$ & $W^I((-1),(1))$ %Coefficient
  \\
\hline
$B_{2k+1}$ & $B_{2k}$ & $\Delta_L$ & $\dfrac{2}{2}\dfrac{2\cdot 4\cdots (4k-2)\cdot 4k}{2\cdot 4\cdots (4k-2)\cdot 2k}=2$ \\
  $C_{2k+1}$ & $C_{2k}$ & $\Delta_S$ & $\dfrac{2}{2}\dfrac{2\cdot 4\cdots (4k-2)\cdot 4k}{2\cdot 4\cdots (4k-2)\cdot 2k}=2$ \\
  $G_2$ & $A_1$ (long) & $\Delta_L$ & $\dfrac{2}{1}\dfrac{2}{2}=2$ \\
  $G_2$ & $A_1$ (short) & $\Delta_S$ & $\dfrac{2}{1}\dfrac{2}{2}=2$ \\
\hline
\end{tabular}
% \begin{tabular}{c|c|c|c|c}
%   Type of $\Delta$ & Type of $\Delta_I$ & Coefficient & $I$ & $J$ \\
% \hline
% $B_{2k+1}$ & $B_{2k}$ & $2\cdot2\cdot 4\cdots 4k/2\cdot2\cdot 4\cdots (4k-2)\cdot 2k=2$ & $I=\{2,\ldots,2k+1\}$ & $J=\{2k+1\}$ \\
% $C_{2k+1}$ & $C_{2k}$ & $2\cdot2\cdot 4\cdots 4k/2\cdot2\cdot 4\cdots (4k-2)\cdot 2k=2$ & $I=\{2,\ldots,2k+1\}$ & $J=\{2k+1\}$ \\
% $G_2$ & $A_1$ & $2\cdot2/2=2$ & $I=\{1\}$ & $J=\{2\}$ \\
% \hline
% \end{tabular}
\end{center}
%where in the $G_2$ case, 
\end{theorem}
\begin{proof}
  The root systems which have roots of two (long and short) lengths are of type $B_r,C_r,F_4,G_2$.
  
When $\Delta$ is of type $F_4$, by Corollary \ref{cor:twolen}, $\Delta_1$ is of
type $D_4$.
Since the degrees of $\Delta(D_4)$ are all even, we see that $W^I((-1),(1))$ vanishes for any choice of $I$ in $\Delta(F_4)$.

  Assume that $\Delta$ is of type $B_{r}$. Then by Corollary \ref{cor:twolen},
$\Delta_1$ must be $\Delta_L$ and
$r$ must be odd because otherwise all the degrees are even.
When $\Delta_1=\Delta_L$, $(r-1)$ degrees are even.
Hence $\Delta_I$ must be of type $B_{r-1}$ whose degrees are all even, which is the 
unique choice such that $|K_1|=|K_{1I}|$.

In the case of $C_r$, a similar argument works and we obtain the result.

Assume that $\Delta$ is of type $G_2$. Then $\Delta_1$ is either $\Delta_L$ or $\Delta_S$.
However $\Delta_I$ must be determined uniquely
so that
$|K_1|=|K_{1I}|$, that is, $\Delta_I\subset\Delta_1$.
\end{proof}

For example, see the following diagram for the pair $(B_{2k+1},B_{2k})$, where the set of enclosed nodes is $\Psi_I$.

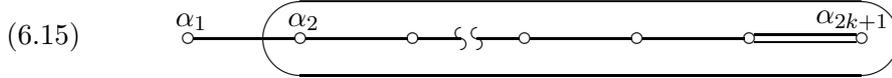
\begin{equation}
\setlength{\unitlength}{1pt}\begin{picture}(0,0)
\put(149,2){\oval(238,28)}
%\put(44,2){\circle{26}}%
%\put(254,2){\circle{26}}%
\end{picture}
\VN{\alpha_1}\E\VN{\alpha_2}\E\V\EO\V\E\V\E\V\EB{\E}{\E}\VN{\!\!\!\!\!\!\alpha_{2k+1}}
\end{equation}
\bigskip

%%%%%%%%%%%%%%%%%%%%%%%%%%%%%%%%%%%%%%%%%%%%%%%%%%%%%%%%%%%%%%%%%%%%%%%%%%%%%%%%%%%%%%%%%%%%%%
\section{Examples}\label{sec7}
%%%%%%%%%%%%%%%%%%%%%%%%%%%%%%%%%%%%%%%%%%%%%%%%%%%%%%%%%%%%%%%%%%%%%%%%%%%%%%%%%%%%%%%%%%%%%%%

In this section, we give several explicit examples of the functional relations among zeta-functions of root systems. For the case ${\bf y}=(0)$, we write $\zeta_r({\bf s};\Delta)$ instead of $\zeta_r({\bf s},{\bf y};\Delta)$ for short.
First we consider the $A_r$-type. 

\begin{example}\label{ex-1}
Set $r=2$, $s_{23}\in \mathbb{R}_{>1}$ and $(y_1,y_2)=(0,0)$ in \eqref{Fq-rel-1}. By considering its real part, we obtain
\begin{align}
& \zeta_2(k_{12},k_{13},s_{23};A_2)+(-1)^{k_{12}}\zeta_2(k_{12},s_{23},k_{13};A_2)\label{A2-form}\\
& \quad +(-1)^{k_{12}+k_{13}}\zeta_2(s_{23},k_{12},k_{13};A_2)\notag\\
& =2\sum_{j_2=0}^{[k_{12}/2]}(-1)^{k_{12}}\binom{k_{12}+k_{13}-1-2j_2}{k_{13}-1}\zeta(2j_2)\zeta(k_{12}+k_{13}+s_{23}-2j_2)\notag\\
& \ +2\sum_{j_3=0}^{[k_{13}/2]}(-1)^{k_{13}}\binom{k_{12}+k_{13}-1-2j_3}{k_{12}-1}\zeta(2j_3)\zeta(k_{12}+k_{13}+s_{23}-2j_3),\notag
\end{align}
where we use the well-known formula
$$\zeta(2k)=-\frac{B_{2k}(0)(2\pi \sqrt{-1})^{2k}}{2(2k)!}\quad (k\in \mathbb{Z}_{\geq 0}).$$
Dividing the both sides by $(-1)^{k_{12}}$, 
we recover the known result in \cite[Theorem 3.1]{KMTNicchuu}, which is equivalent to \cite[Theorem 4.5]{Tsu07} (see also \cite{Nak06}). Here we note that 
$$\zeta_2(s_{12},s_{13},s_{23};A_2)=\sum_{m_1=1}^\infty \sum_{m_2=1}^\infty \frac{1}{m_1^{s_{12}}{(m_1+m_2)}^{s_{13}}m_2^{s_{23}}},$$
where the order of indices is slightly different from that in \cite{KMTNicchuu}.
\end{example}

\begin{example}\label{ex-2}
Set $r=3$, $(s_{23},s_{24},s_{34})\in (\mathbb{R}_{>1})^3$ and $(y_1,y_2,y_3)=(0,0,0)$ in \eqref{Fq-rel-1}. By considering its real part, we obtain
\begin{align}
& \zeta_3(k_{12},k_{13},k_{14},s_{23},s_{24},s_{34};A_3) \label{A3-fq-01}\\
& +(-1)^{k_{12}}\zeta_3(k_{12},s_{23},s_{24},k_{13},k_{14},s_{34};A_3) \notag\\
& +(-1)^{k_{12}+k_{13}}\zeta_3(s_{23},k_{12},s_{24},k_{13},s_{34},k_{14};A_3) \notag\\
& +(-1)^{k_{12}+k_{13}+k_{14}}\zeta_3(s_{23},s_{24},k_{12},s_{34},k_{13},k_{14};A_3) \notag\\
&=2\sum_{j_2=0}^{[k_{12}/2]}\sum_{l_3,l_4\geq 0 \atop l_3+l_4=k_{12}-2j_2}(-1)^{k_{12}}\binom{k_{13}+l_3-1}{l_3}\binom{k_{14}+l_4-1}{l_4}\notag\\
& \qquad \times \zeta(2j_2)\zeta_2(s_{23}+k_{13}+l_3,s_{24}+k_{14}+l_4,s_{34};A_2)\notag\\
&=2\sum_{j_3=0}^{[k_{13}/2]}\sum_{l_2,l_4\geq 0 \atop l_2+l_4=k_{13}-2j_3}(-1)^{k_{12}+l_4}\binom{k_{12}+l_2-1}{l_2}\binom{k_{14}+l_4-1}{l_4}\notag\\
& \qquad \times \zeta(2j_3)\zeta_2(s_{23}+k_{12}+l_2,s_{24},s_{34}+k_{14}+l_4;A_2)\notag\\
&=2\sum_{j_4=0}^{[k_{14}/2]}\sum_{l_2,l_3\geq 0 \atop l_2+l_3=k_{14}-2j_4}(-1)^{k_{12}+k_{13}}\binom{k_{12}+l_2-1}{l_2}\binom{k_{13}+l_3-1}{l_3}\notag\\
& \qquad \times \zeta(2j_4)\zeta_2(s_{23},s_{24}+k_{12}+l_2,s_{34}+k_{13}+l_3;A_2).\notag
\end{align}
This formula is apparently different from the previous result in \cite[Theorem 4.4]{KMTKyushu}. 
For example, we obtain from \eqref{A3-fq-01} with $(k_{ij})=(2)$ and $(s_{ij})=(2)$ that
\begin{align*}
4\zeta_3(2,2,2,2,2,2;A_3) & =2\zeta(2)\left\{2\zeta_2(4,4,2;A_2)+\zeta_2(4,2,4;A_2)\right\}\\
& -6\zeta_2(6,4,2;A_2)-6\zeta_2(6,2,4;A_2)-8\zeta_2(5,5,2;A_2)\\
& +4\zeta_2(5,2,5;A_2)-6\zeta_2(4,6,2;A_2).
\end{align*}
On the other hand, we already obtained from \cite[Eq.\,(4.28)]{KMTKyushu} that
\begin{align*}
4\zeta_3(2,2,2,2,2,2;A_3) & =8\zeta(2)\left\{\zeta_2(4,4,2;A_2)+\zeta_2(3,5,2;A_2)\right\}\\
& -12\zeta_2(6,4,2;A_2)-12\zeta_2(5,5,2;A_2)-6\zeta_2(4,6,2;A_2).
\end{align*}
These two right-hand sides are apparently different.
However we can check that both of the right-hand sides are equal to 
$887\pi^{12}/3831077250$ by the method of partial fraction decompositions (see \cite[Example 4.5]{KMTKyushu}). 
\end{example}

Next we consider the $C_r$-type.

\begin{example}
We define zeta-functions of root systems of $C_2$-type and $C_3$-type by
\begin{align*}
& \zeta_2(s_1,s_2,s_3,s_4;C_2) =\sum_{m,n\geq 1}\frac{1}{m^{s_1}n^{s_2}(m+n)^{s_3}(m+2n)^{s_4}},\\
& \zeta_3(s_1,s_2,s_3,s_4,s_5,s_6,s_7,s_8,s_9;C_3)\\
& =\sum_{m_1,m_2,m_3\geq 1}\frac{1}{m_1^{s_1}m_2^{s_2}m_3^{s_3}(m_1+m_2)^{s_4}(m_2+m_3)^{s_5}(m_2+2m_3)^{s_6}}\\
& \quad \times \frac{1}{(m_1+m_2+m_3)^{s_7}(m_1+m_2+2m_3)^{s_8}(m_1+2m_2+2m_3)^{s_9}}.
\end{align*}
These were already studied in \cite{KMTWitten2,KMTLondon,KMTNicchuu}. 
Here we give explicit forms of functional relations among them as follows. 

Let $r=3$, $\Delta=\Delta(C_3)$, $I=\{ 2,3\}$ and $(y_1,y_2,y_3)=(0,0,0)$ 
in \eqref{3-41}. Then we have
  \begin{multline}\label{C3-gene-F}
  F((t_1,t_{\pm 2},t_{\pm 3}),\mathbf{0},(m_2,m_3);\{2,3\};C_3)  
  \\
  \begin{aligned}
    &= %j=2
    \frac{t_{-3}}{t_{-3}-t_{-2}-2\pi\sqrt{-1}m_2}
    \frac{t_{+2}}{t_{+2}-t_{-2}-2\pi\sqrt{-1}(2(m_2+m_{3}))}
    \\
    &\quad\times
    \frac{t_{+3}}{t_{+3}-t_{-2}-2\pi\sqrt{-1}(m_2+2m_{3})}
    \frac{t_1}{t_1-t_{-2}-2\pi\sqrt{-1}(m_2+m_{3})}
    \frac{t_{-2}}{e^{t_{-2}}-1}
    \\
    &+ %j=3
    \frac{t_{-2}}{t_{-2}-t_{-3}+2\pi\sqrt{-1}m_2}
    \frac{t_{+2}}{t_{+2}-t_{-3}-2\pi\sqrt{-1}(m_2+2m_3)}
    \\
    &\quad\times
    \frac{t_{+3}}{t_{+3}-t_{-3}-2\pi\sqrt{-1}(2m_3)}
    \frac{t_1}{t_1-t_{-3}-2\pi\sqrt{-1}m_3}
    \frac{t_{-3}}{e^{t_{-3}}-1}
    \\
    &+ %j=2
    \frac{t_{-2}}{t_{-2}-t_{+2}+2\pi\sqrt{-1}(2(m_2+m_{3}))}
    \frac{t_{-3}}{t_{-3}-t_{+2}+2\pi\sqrt{-1}(m_2+2m_3))}
    \\
    &\quad\times
    \frac{t_{+3}}{t_{+3}-t_{+2}+2\pi\sqrt{-1}m_2}
    \frac{t_1}{t_1-t_{+2}+2\pi\sqrt{-1}(m_2+m_{3})}
    \frac{t_{+2}}{e^{t_{+2}}-1}
    \\
    &+ %j=3
    \frac{t_{-2}}{t_{-2}-t_{+3}+2\pi\sqrt{-1}(m_2+2m_3)}
    \frac{t_{-3}}{t_{-3}-t_{+3}+2\pi\sqrt{-1}(2m_3)}
    \\
    &\quad\times
    \frac{t_{+2}}{t_{+2}-t_{+3}-2\pi\sqrt{-1}m_2}
    \frac{t_1}{t_1-t_{+3}+2\pi\sqrt{-1}m_3}
    \frac{t_{+3}}{e^{t_{+3}}-1}
    \\
    &+
    \frac{t_{-2}}{t_{-2}-t_{1}+2\pi\sqrt{-1}(m_2+m_3)}
    \frac{t_{-3}}{t_{-3}-t_{1}+2\pi\sqrt{-1}m_3}
    \\
    &\quad\times
    \frac{t_{+2}}{t_{+2}-t_{1}-2\pi\sqrt{-1}(m_2+m_3)}
    \frac{t_{+3}}{t_{+3}-t_{1}-2\pi\sqrt{-1}m_3}
    \frac{t_{1}}{e^{t_{1}}-1}.
\end{aligned}
\end{multline}
Hence we can compute $P(\mathbf{k},\mathbf{y},\lambda;I;C_3)$ and give some functional relations from Theorem \ref{thm:main1}. 
For example, we obtain
\begin{multline}
  P((2,1,1,1,1),\mathbf{0},(m_2,m_3);\{2,3\};C_3)=
  \\
  \begin{aligned}
    &\frac{1}{32 \pi ^6 m_2^2 m_3^3 (m_2+2 m_3)}+\frac{1}{32 \pi ^6 m_2^2 (m_2+m_3)^3 (m_2+2 m_3)}
    \\
    &-\frac{1}{32 \pi ^6 m_3^4 (m_2+m_3)^2}-\frac{5}{64 \pi ^6 m_2 m_3^4 (m_2+2 m_3)}
    \\
    &-\frac{1}{32 \pi ^6 m_2 m_3^3 (m_2+2 m_3)^2}-\frac{1}{32 \pi ^6 m_3^2 (m_2+m_3)^4}+\frac{1}{96 \pi ^4 m_3^2 (m_2+m_3)^2}
    \\
    &+\frac{5}{64 \pi ^6 m_2 (m_2+m_3)^4 (m_2+2 m_3)}+\frac{1}{32 \pi ^6 m_2 (m_2+m_3)^3 (m_2+2 m_3)^2}.
  \end{aligned}
\end{multline}
Hence we have
\begin{align}
&                  \zeta_3(1,s,t,1,u,v,2,1,1;C_3)
        -    \zeta_3(1,1,t,s,2,1,u,v,1;C_3) \label{C3-fr-1}\\
& +     \zeta_3(s,1,2,1,t,1,u,1,v;C_3)
 +     \zeta_3(s,1,2,1,t,1,u,1,v;C_3)\notag\\
&-   \zeta_3(1,1,t,s,2,1,u,v,1;C_3)
 + \zeta_3(1,s,t,1,u,v,2,1,1;C_3)\notag\\
    & =
    (-1)^5\frac{(2\pi\sqrt{-1})^{6}}{2!1!1!1!1!}
    \sum_{m_2=1}^\infty
    \sum_{m_3=1}^\infty \frac{P((2,1,1,1,1),\mathbf{0},(m_2,m_3);\{2,3\};C_3)}{m_2^s m_3^t (m_2+m_3)^u (m_2+2m_3)^v}\notag\\
& =\zeta_2(s+2,t+3,u,v+1;C_2) +\zeta_2(s+2,t,u+3,v+1;C_2)  \notag\\
& \quad -\zeta_2(s,t+4,u+2,v;C_2)-\frac{5}{2}\zeta_2(s+1,t+4,u,v+1;C_2) \notag\\
& \quad -\zeta_2(s+1,t+3,u,v+2;C_2) -\zeta_2(s,t+2,u+4,v;C_2)\notag\\
& \quad  +\frac{\pi^2}{3}\zeta_2(s,t+2,u+2,v;C_2) +\frac{5}{2}\zeta_2(s+1,t,u+4,v+1;C_2) \notag\\
& \quad +\zeta_2(s+1,t,u+3,v+2;C_2).\notag
\end{align}
Setting $(s,u,v)=(1,2,1)$, we obtain
\begin{align}
& 2\zeta_3(1,1,2,1,t,1,2,1,1;C_3)\label{C3-fr-2}\\
& =\zeta_2(3,t+3,2,2;C_2) +\zeta_2(3,t,5,2;C_2) -\zeta_2(1,t+4,4,1;C_2) \notag\\
& \quad -\frac{5}{2}\zeta_2(2,t+4,2,2;C_2) -\zeta_2(2,t+3,2,3;C_2) \notag\\
& \quad  -\zeta_2(1,t+2,6,1;C_2)+\frac{\pi^2}{3}\zeta_2(1,t+2,4,1;C_2)  \notag\\
& \quad +\frac{5}{2}\zeta_2(2,t,6,2;C_2)+\zeta_2(2,t,5,3;C_2).\notag
\end{align}
In particular when $t=2$, this is an example of the assertion in 
Section \ref{sec-6-3} corresponding to $\Delta=\Delta(C_3)$, $\Delta_I=\Delta(C_2)$ with $I=\{ 1\}$ and $W^I((-1),(1))=2$, because the left-hand does not vanish and its coefficient is $2$.

Moreover we here recall the known fact that 
\begin{equation}
\zeta_2(a,b,c,d;C_2)\in \mathbb{Q}\left[\pi^2,\{\zeta(2j+1)\}_{j\in \mathbb{N}}\right]\label{C2-Parity}
\end{equation}
for $a,b,c,d\in \mathbb{N}$ with $2 \nmid (a+b+c+d)$, 
which was given by the third-named author (see \cite{Tsu04}). If we set $t=2k-1$ $(k\in \mathbb{N})$ in \eqref{C3-fr-2}, we obtain from \eqref{C2-Parity} that 
\begin{equation}
\begin{split}
&\zeta_3(1,1,2,1,2k-1,1,2,1,1;C_3)\in \mathbb{Q}\left[\pi^2,\{\zeta(2j+1)\}_{j\in \mathbb{N}}\right]
\end{split}
\label{C3-Parity}
\end{equation}
for $k\in \mathbb{N}$. Setting $t=1,3$, we have
\begin{align*}
& 2\zeta_3(1,1,2,1,1,1,2,1,1;C_3)= \frac{3}{20} \zeta(7)\pi^4 - \frac{233}{16}\zeta(9)\pi^2 + \frac{4135}{32} \zeta(11),\\
& 2\zeta_3(1,1,2,1,3,1,2,1,1;C_3) = -\frac{7}{15} \zeta(9)\pi^4 + \frac{681}{16} \zeta(11)\pi^2 - \frac{5995}{16} \zeta(13).
\end{align*}
More generally, we can prove that 
\begin{equation}
\begin{split}
&\zeta_3(2q-1,2q-1,2r,2p-1,2k-1,2q-1,2r,2p-1,2q-1;C_3)\\
& \qquad \in \mathbb{Q}\left[\pi^2,\{\zeta(2j+1)\}_{j\in \mathbb{N}}\right]
\end{split}
\label{C3-Parity}
\end{equation}
for $p,q,r,k\in \mathbb{N}$. 
For example, we can give
\begin{align*}
& 2 \zeta_3(1,1,2,3,1,1,2,3,1;C_3)\\
&\quad = \frac{1}{10080} \zeta(7)\pi^8 + \frac{7}{180} \zeta(9)\pi^6 
  + \frac{13217}{576} \zeta(11)\pi^4 \\
& \qquad - \frac{4132493}{1536}\zeta(13)\pi^2 
  + \frac{24864015}{1024} \zeta(15),\\
&2 \zeta_3(1,1,2,3,5,1,2,3,1;C_3)\\
&\quad = \frac{43}{90720} \zeta(11)\pi^8 + \frac{859}{2520} \zeta(13)\pi^6  
  + \frac{140051}{576} \zeta(15)\pi^4 \\
& \qquad - \frac{42288073}{1536} \zeta(17)\pi^2 
  + \frac{253652169}{1024} \zeta(19).
\end{align*}

\end{example}

Finally we consider the $G_2$-type.

\begin{example}
We define the zeta-function of root system of $G_2$-type by
\begin{align*}
& \zeta_2(s_1,s_2,s_3,s_4,s_5,s_6;G_2) \\
& \quad =\sum_{m,n\geq 1}\frac{1}{m^{s_1}n^{s_2}(m+n)^{s_3}(m+2n)^{s_4}(m+3n)^{s_5}(2m+3n)^{s_6}}.
\end{align*}
This was already studied in \cite{KMTWitten4,KMTWitten5}. In fact, it follows from \cite[Theorem 6.1]{KMTWitten5} that 
for $p,q,r,u,v \in \mathbb{N}$,
\begin{equation}
\begin{split}
& \zeta_2(p,s,q,r,u,v;G_2)+(-1)^p \zeta_2(p,q,s,r,v,u;G_2)\\
& \ \ +(-1)^{p+q}\zeta_2(v,q,r,s,p,u;G_2)+(-1)^{p+q+v}\zeta_2(v,r,q,s,u,p;G_2) \\
& \ \ +(-1)^{p+q+r+v}\zeta_2(u,r,s,q,v,p;G_2)\\
& \ \ +(-1)^{p+q+r+u+v}\zeta_2(u,s,r,q,p,v;G_2)
\end{split}
\label{G2-FE}
\end{equation}
can be expressed in terms of the Riemann zeta-function. In particular, setting $(p,s,q,r,u,v)=(2a,2b-1,2b-1,2b-1,2a,2a)$ for $a,b\in \mathbb{N}$, we can compute \eqref{G2-FE} as 
$$2\zeta_2(2a,2b-1,2b-1,2b-1,2a,2a;G_2).$$
This is an example of the assertion in 
Section \ref{sec-6-3} corresponding to $\Delta=\Delta(G_2)$, $\Delta_I=\Delta(A_1)$ with $I=\{ 2\}$ and $W^I((-1),(1))=2$. 
For example, we can obtain from \cite[Theorem 6.1]{KMTWitten5} that 
\begin{align*}
2 \zeta_2(2,1,1,1,2,2;G_2)
  &  =-\frac{187}{972}\zeta(7)\pi^2+\frac{11149}{5832}\zeta(9),\\
2 \zeta_2(4,1,1,1,4,4;G_2)
  &  =-\frac{15337}{4723920}\zeta(11)\pi^4-\frac{157303}{2834352}\zeta(13)\pi^2\\
& \quad\ \   +\frac{14696765}{17006112}\zeta(15),\\
2 \zeta_2(2,3,3,3,2,2;G_2)
  &  =-\frac{16171}{3888}\zeta(13)\pi^2+\frac{957697}{23328}\zeta(15).
\end{align*}
\end{example}

\ 

%%%%%%%%%%%%%%%%%%%%%%%%%%%%%%%%%%%%%%%%%%%%%%%%%%%%%%%%%%%%%%%%%%%%%%%%%%%%%%%%%%%%%%%%%%%%%
\section{Proof of Theorem \ref{thm:main1}}\label{sec8}
%Theorema Egregium}
%%%%%%%%%%%%%%%%%%%%%%%%%%%%%%%%%%%%%%%%%%%%%%%%%%%%%%%%%%%%%%%%%%%%%%%%%%%%%%%%%%%%%%%%%%%%%%
Now we proceed to prove fundamental results stated in Section \ref{sec2}.
In this section we prove Theorem \ref{thm:main1}.
First we review a result proved in \cite{KMT2014}.

Denote the usual inner product on $\mathbb{R}^d$ by
$\langle\cdot,\cdot\rangle'$.
Regard $f=(\vec{f},\Dot{f})\in \mathbb{R}^d\times\mathbb{C}$ ($\vec{f}\in \mathbb{R}^d$, 
$\Dot{f}\in \mathbb{C}$)
as an affine linear functional on $\mathbb{R}^d$ by 
$f(\mathbf{u})=\langle\vec{f},\mathbf{u}\rangle'+\Dot{f}$ for $\mathbf{u}\in \mathbb{R}^d$.
For $X\subset \mathbb{R}^d$, denote by $\langle X\rangle$ the $\mathbb{Z}$-span of $X$.
For $Y\subset \mathbb{R}^d\times\mathbb{C}$, put
$\vec{Y}=\{\vec{f}\;|\;f=(\vec{f},\Dot{f})\in Y\}$.

Let $\Lambda$ be a finite subset of $(\mathbb{Z}^d\setminus\{0\})\times\mathbb{C}$
such that the rank of $\langle\vec{\Lambda}\rangle$ is $d$.   Denote by
$\widetilde{\Lambda}$ the set of all $f\in\Lambda$ such that the rank of 
$\langle\vec{\Lambda}\setminus\{\vec{f}\}\rangle$ is less than $d$.
For each $f\in\Lambda$ we associate a number $k_f\in\mathbb{N}$, and put
$\mathbf{k}=(k_f)_{f\in\Lambda}$.     For $k\in\mathbb{N}$, define
$\Lambda_k=\{f\in\Lambda\;|\;k_f=k\}$.    Then $\Lambda=\bigcup_{k\geq 1}\Lambda_k$.

Let $\mathbf{y}'\in \mathbb{R}^d$, and define
\begin{align}\label{rev-1}
S(\mathbf{k},\mathbf{y}';\Lambda)=\lim_{N\to\infty}\sum_
{\substack{\mathbf{u}=(u_1,\ldots,u_d)\in\mathbb{Z}^d \\ |u_j|\leq N\;(1\leq j\leq d)\\
f(\mathbf{u})\neq 0\;(f\in\Lambda)}}e^{2\pi\sqrt{-1}\langle\mathbf{y}',\mathbf{u}\rangle'}
\prod_{f\in\Lambda}\frac{1}{f(\mathbf{u})^{k_f}}.
\end{align}
Under the assumption $(\#)$ in the statement of Theorem \ref{thm:main1}, we see that
$\widetilde{\Lambda}\cap\Lambda_1=\emptyset$.    Therefore by \cite[Theorem 2.2]{KMT2014},
the right-hand side of \eqref{rev-1} converges %absolutely
(and defines 
$S(\mathbf{k},\mathbf{y}';\Lambda)$) for any $\mathbf{y}'\in \mathbb{R}^d$.

Let $\mathscr{B}(\Lambda)$ be the collection of all subsets 
$B=\{f_1,\ldots,f_d\}\subset\Lambda$ such that $\vec{B}$ forms a basis of $\mathbb{R}^d$.
We write the dual basis of $\vec{B}=\{\vec{f}_1,\ldots,\vec{f}_d\}$ as 
$\vec{B}^*=\{\vec{f}_1^B,\ldots,\vec{f}_d^B\}$.
For any $\mathbf{y}'\in \mathbb{R}^d$ and any 
$\mathbf{t}=(t_f)_{f\in\Lambda}\in\mathbb{C}^{|\Lambda|}$,
define
\begin{multline}
  \label{eq:exp_F-rev}
  F(\mathbf{t},\mathbf{y}';\Lambda)
  =
  \sum_{B\in\mathscr{B}(\Lambda)}
  \left(\prod_{g\in \Lambda\setminus B}
  \frac{t_g}
  {t_g-2\pi\sqrt{-1}\Dot{g}-\sum_{f\in B}
    (t_f-2\pi\sqrt{-1}\Dot{f})\langle \vec{g},\vec{f}^B\rangle'}\right)
  \\
  \times
  \frac{1}{\abs{\mathbb{Z}^d/\langle\vec{B}\rangle}}
  \sum_{\mathbf{w}\in \mathbb{Z}^d/\langle\vec{B}\rangle}
  \prod_{f\in B}
  \frac{t_f \exp((t_f-2\pi\sqrt{-1}\Dot{f})
    \{\mathbf{y}'+\mathbf{w}\}_{B,f}^{\prime})}{\exp(t_f-2\pi\sqrt{-1}\Dot{f})-1},
\end{multline}
where $\{\cdot\}_{B,f}^{\prime}$ is the multi-dimensional fractional part in the sense of
\cite[Section 2]{KMT2014} (using the standard inner product on $\mathbb{R}^d$, which is
written as $\langle\cdot,\cdot\rangle$ in \cite{KMT2014}, but in the present paper it is
written as $\langle\cdot,\cdot\rangle'$), 
and define $C(\mathbf{k},\mathbf{y}';\Lambda)$ by the expansion
\begin{equation}
\label{eq:def_F-rev}
  F(\mathbf{t},\mathbf{y}';\Lambda)=
  \sum_{\mathbf{k}\in \mathbb{N}_0^{\abs{\Lambda}}}C(\mathbf{k},\mathbf{y}';\Lambda)
  \prod_{f\in \Lambda}
  \frac{t_f^{k_f}}{k_f!}.
\end{equation}
Then \cite[Theorem 2.5]{KMT2014} asserts
\begin{align}\label{rev-2}
S(\mathbf{k},\mathbf{y}';\Lambda)=\left(\prod_{f\in\Lambda}\frac{-(2\pi\sqrt{-1})^{k_f}}
{k_f!}\right)C(\mathbf{k},\mathbf{y}';\Lambda)
\end{align}
for any $\mathbf{k}=(k_f)_{f\in\Lambda}\in\mathbb{N}^{|\Lambda|}$ and any
$\mathbf{y}'\in \mathbb{R}^d$.    
This is the key formula for the proof of Theorem \ref{thm:main1}.

\begin{proof}[Proof of Theorem \ref{thm:main1}]
For $\lambda\in P_{I+}$ and $\mu\in P_{I^c}$,
the relation
$\lambda+\mu\not\in H_{\Delta_+^{\vee}}=H_{\Delta_{I+}^\vee}\cup H_{\Delta^{*\vee}}$
is equivalent to
\begin{equation}
  \lambda\not\in H_{\Delta_{I+}^\vee},
\qquad\text{and}\qquad
  \lambda+\mu\not\in H_{\Delta^{*\vee}}.
\end{equation}
Hence for $\mathbf{s}\in\mathbb{C}^{|\Delta_+|}$ and $\mathbf{y}\in V$ we have
\begin{multline}\label{S-decomp}
    S(\mathbf{s},\mathbf{y};I;\Delta)
    \\
    \begin{aligned}
      &=
    \sum_{\lambda\in P_{I+}}
    \sum_{\substack{\mu\in P_{I^c}\\\lambda+\mu\not\in H_{\Delta_+^{\vee}}}}
    \biggl(
    \prod_{\alpha\in\Delta_{I+}}
    \frac{1}{\langle\alpha^\vee,\lambda+\mu\rangle^{s_\alpha}}
    \biggr)
    e^{2\pi\sqrt{-1}\langle \mathbf{y},\lambda\rangle}
    \biggl(
    e^{2\pi\sqrt{-1}\langle \mathbf{y},\mu\rangle}
    \prod_{\alpha\in \Delta^*} \frac{1}{\langle\alpha^\vee,\lambda+\mu\rangle^{k_\alpha}}\biggr)
    \\
    &=
    \sum_{\substack{\lambda\in P_{I++}}}
    \biggl(
     \prod_{\alpha\in\Delta_{I+}}
    \frac{1}{\langle\alpha^\vee,\lambda\rangle^{s_\alpha}}
    \biggr)
    e^{2\pi\sqrt{-1}\langle \mathbf{y},\lambda\rangle}
S_{\lambda}(\mathbf{k},\mathbf{y};I;\Delta),
  \end{aligned}
\end{multline}
where
\begin{align}\label{def_Sigma_lambda}
S_{\lambda}(\mathbf{k},\mathbf{y};I;\Delta)=
        \sum_{\substack{\mu\in P_{I^c}\\\lambda+\mu\not\in H_{\Delta^{*\vee}}}}
    \biggl(
    e^{2\pi\sqrt{-1}\langle \mathbf{y},\mu\rangle}
    \prod_{\alpha\in \Delta^*} \frac{1}{\langle\alpha^\vee,\lambda+\mu\rangle^{k_\alpha}}\biggr). 
\end{align}

We apply the aforementioned result of \cite{KMT2014} to find the generating function of 
$S_{\lambda}(\mathbf{k},\mathbf{y};I;\Delta)$.
For this purpose, we need to identify each symbol appearing here and in \cite{KMT2014}.
We identify $\mathbb{R}^{d}\simeq
\bigoplus_{i\in I^c}\mathbb{R}\alpha_i^\vee$
and
through this identification,
we define the projection $\pi_{I^c}:V\to \mathbb{R}^d$ 
by
\begin{equation}
  \pi_{I^c}(v)=(\langle v,\lambda_i\rangle)_{i\in I^c} \quad (v\in V).
\end{equation}
For $\alpha\in\Delta^*$,
we set the affine linear form 
\begin{equation}
  f_\alpha(\mu)=\langle\alpha^\vee,\mu\rangle+\langle\alpha^\vee,\lambda\rangle
\end{equation}
for $\mu\in\bigoplus_{i\in I^c}\mathbb{R}\lambda_i$.

First we note that
{\it $f_\alpha$ is naturally regarded as an affine linear form on $\mathbb{R}^d$ via 
$\mathbb{R}^d\simeq\bigoplus_{i\in I^c}\mathbb{R}\lambda_i$}.
In fact, for $\alpha^\vee=\sum_{i\in I^c}a_i\alpha_i^\vee$ and 
$\mu=\sum_{i\in I^c}b_i \lambda_i$, we have
\begin{equation}
\label{eq:amu}
  \begin{split}
    \langle\alpha^\vee,\mu\rangle&=
    \sum_{i\in I^c}a_ib_i=
    \langle(a_i)_{i\in I^c},(b_i)_{i\in I^c}\rangle'
    =
    \langle(\langle\alpha^\vee,\lambda_i\rangle)_{i\in I^c},(b_i)_{i\in I^c}\rangle'
    \\
    &=
    \langle\pi_{I^c}(\alpha^\vee),(b_i)_{i\in I^c}\rangle'.
  \end{split}
\end{equation}
Thus we may choose
\begin{equation}\label{choice_Lambda}
  \Lambda=\{f_\alpha=(\vec{f}_\alpha,\Dot{f}_\alpha)~|~\alpha\in\Delta^*\}
  \subset (\mathbb{Z}^d\setminus\{0\})\times\mathbb{Z},
\end{equation}
where $\vec{f}_\alpha=\pi_{I^c}(\alpha^\vee)$ and $\Dot{f}_\alpha=\langle\alpha^\vee,\lambda\rangle$.

Next, let $C\subset\Delta^*$ with $|C|=d$ and
$B=\{f_\alpha~|~\alpha\in C\}\subset\Lambda$.
Secondly we show that
$B\in\mathscr{B}(\Lambda)$ {\it if and only if} 
$\mathbf{V}=C\cup \Psi_I\in\mathscr{V}_I$.
Since 
the matrix of the pairings $\langle\beta^\vee,\lambda_i\rangle$ for $\beta\in C\cup \Psi_I$ and $1\leq i\leq r$ is written as the form
\begin{equation}
\label{eq:matbl}
  \begin{pmatrix}
    (\langle\beta^\vee,\lambda_i\rangle)_{\beta\in C,i\in I^c} & (\langle\beta^\vee,\lambda_i\rangle)_{\beta\in C,i\in I}
\\
(\langle\beta^\vee,\lambda_i\rangle)_{\beta\in\Psi_I,i\in I^c} & (\langle\beta^\vee,\lambda_i\rangle)_{\beta\in\Psi_I,i\in I}
  \end{pmatrix}
  =
\begin{pmatrix}
    A & * \\ 0 & \id_{|I|}
  \end{pmatrix},
\end{equation}
it is sufficient to show
that
that $A$ is regular if and only if
$B\in\mathscr{B}(\Lambda)$.
The definition of $\mathscr{B}(\Lambda)$ implies that
$B\in\mathscr{B}(\Lambda)$ if and only if
\begin{equation}
A'=(\langle\vec{f}_\alpha,e_i\rangle')_{\alpha\in C,i\in I^c}    
\end{equation}
is a regular matrix, where $\{e_i\}_{i\in I^c}$ is the standard basis of $\mathbb{R}^d$. 
Here
by \eqref{eq:amu}, we have
\begin{equation}
  \langle\vec{f}_\alpha,e_i\rangle'=\langle \pi_{I^c}(\alpha^{\vee}),e_i\rangle'=
  \langle\alpha^\vee,\lambda_i\rangle
\end{equation}
and so $A'$ is rewritten as
\begin{equation}
A'=(\langle\alpha^\vee,\lambda_i\rangle)_{\alpha\in C,i\in I^c}=A.
\end{equation}
Thus we showed the desired claim.
When these equivalent conditions hold, we write $C=\mathbf{V}_I$, which agrees with the notation in Section \ref{sec2}.
In this case
%for $\mathbf{V}=C\cup \Psi_I\in\mathscr{V}_I$,
we have
\begin{equation}
  \begin{split}
    Q^\vee/L(\mathbf{V}^\vee)&=
\Bigl(\bigoplus_{i=1}^r\mathbb{Z}\alpha_i^\vee\Bigr)/
\Bigl(\bigoplus_{\beta\in \mathbf{V}_I}\mathbb{Z}\beta^\vee\oplus
\bigoplus_{i\in I}\mathbb{Z}\alpha_i^\vee\Bigr)
\\
&\simeq
\pi_{I^c}\Bigl(\bigoplus_{i\in I^c}\mathbb{Z}\alpha_i^\vee\Bigr)/
\pi_{I^c}\Bigl(\bigoplus_{\beta\in \mathbf{V}_I}\mathbb{Z}\beta^\vee\Bigr),
\end{split}
\end{equation}
and 
hence $\mathbb{Z}^d/\langle\vec{B}\rangle\simeq Q^\vee/L(\mathbf{V}^\vee)$. % by applying $\pi_{I^c}$.

Thirdly we show that 
$\langle\vec{f}_\gamma,\vec{f}_\alpha^B\rangle'=\langle \gamma^\vee,\mu^{\mathbf{V}}_\alpha\rangle$ for
 $\gamma\in\Delta^*$ and $\alpha\in \mathbf{V}_I$
(and hence $B=\{f_\alpha~|~\alpha\in \mathbf{V}_I\}\in\mathscr{B}(\Lambda)$
and $\mathbf{V}=\mathbf{V}_I\cup \Psi_I$).
It is sufficient to show that for any $v\in V$,
\begin{equation}\label{inner-product-identity}
  \langle \pi_{I^c}(v),\vec{f}_\alpha^B\rangle'=\langle v,\mu^{\mathbf{V}}_\alpha\rangle
\end{equation}
holds.
  Write
  \begin{equation}
    v=\sum_{\beta\in\Psi_I}b_\beta\beta^\vee+\sum_{\beta\in \mathbf{V}_I}c_\beta\beta^\vee.
  \end{equation}
  Then
  \begin{equation}
    \pi_{I^c}(v)=\sum_{\beta\in \mathbf{V}_I}c_\beta\pi_{I^c}(\beta^\vee).
  \end{equation}
  Since
  $\langle\pi_{I^c}(\beta^\vee),\vec{f}^B_\alpha\rangle'=\langle\vec{f}_\beta,\vec{f}^B_\alpha\rangle'=\delta_{\beta\alpha}$ for $\beta\in \mathbf{V}_I$,
  we have
  \begin{equation}
      \langle \pi_{I^c}(v),\vec{f}_\alpha^B\rangle'
      =
      \sum_{\beta\in \mathbf{V}_I}c_\beta\langle\pi_{I^c}(\beta^\vee),\vec{f}_\alpha^B\rangle'
      =c_\alpha
      =\langle v,\mu^{\mathbf{V}}_\alpha\rangle.
  \end{equation}
Therefore
$\langle\vec{f}_\gamma,\vec{f}_\alpha^B\rangle'=\langle \pi_{I^c}(\gamma^\vee),\vec{f}_\alpha^B\rangle'=\langle \gamma^\vee,\mu^{\mathbf{V}}_\alpha\rangle$.

From the above arguments, we see that \eqref{eq:exp_F-rev} under the choice
\eqref{choice_Lambda} now reads as
\begin{equation}
  \begin{split}
    &F(\mathbf{t},\mathbf{y}';\Lambda)\\
    &=
    \sum_{\mathbf{V}\in\mathscr{V}_I}
    \left(\prod_{\gamma\in \Delta^*\setminus\mathbf{V}_I}
    \frac{t_\gamma}
    {t_\gamma
      -2\pi\sqrt{-1}
      \langle \gamma^\vee,\lambda\rangle
      -\sum_{\beta\in\mathbf{V}_I}
      (t_\beta-2\pi\sqrt{-1}
      \langle\beta^\vee,\lambda\rangle)\langle\gamma^\vee,\mu^{\mathbf{V}}_\beta\rangle
    }\right)
    \\
    &\times
    \frac{1}{\abs{Q^\vee/L(\mathbf{V}^\vee)}}
    \sum_{q\in Q^\vee/L(\mathbf{V}^\vee)}
    \prod_{\beta\in\mathbf{V}_I}
    \frac{t_\beta\exp
      ((t_\beta-2\pi\sqrt{-1}
      \langle \beta^\vee,\lambda\rangle)
      \{\mathbf{y}'+q\}^{\prime}_{B,f})}{\exp(t_\beta-2\pi\sqrt{-1}
        \langle \beta^\vee,\lambda\rangle)-1}.
  \end{split}
\end{equation}
In fact, there are the following one-to-one correspondences:
\begin{align*}
B\in\mathscr{B}(\Lambda) &\longleftrightarrow \mathbf{V}\in \mathscr{V}_I,\\
g\in \Lambda\setminus B &\longleftrightarrow \gamma\in \Delta^*\setminus\mathbf{V}_I,\\
f\in B &\longleftrightarrow \beta\in\mathbf{V}_I, \\
\mathbf{w}\in \mathbb{Z}^d/\langle\vec{B}\rangle &\longleftrightarrow 
q\in Q^{\vee}/L(\mathbf{V}^{\vee}), 
\end{align*}
hence $\langle\vec{g},\vec{f}^B\rangle'$ corresponds to
$\langle \vec{f}_{\gamma},\vec{f}_{\beta}^B\rangle'$, the latter being equal to
$\langle \gamma^{\vee}, \mu_{\beta}^{\mathbf{V}}\rangle$ as we have already seen.

From \eqref{eq:amu} we see that
\begin{align}\label{eq:amu2}
\langle\mathbf{y},\mu\rangle=\langle\pi_{I^c}(\mathbf{y}),(b_i)_{i\in I^c}\rangle'
\end{align}
for any $\mathbf{y}\in V$.    Therefore comparing \eqref{rev-1} 
(with $\mathbf{y}'=\pi_{I^c}(\mathbf{y})$) and 
\eqref{def_Sigma_lambda} we find that
\begin{align}\label{S_lamba=S}
S_{\lambda}(\mathbf{k},\mathbf{y};I;\Delta)=S(\mathbf{k},\pi_{I^c}(\mathbf{y});\Lambda),
\end{align}
and also
$\{\pi_{I^c}(\mathbf{y})+q\}_{B,f}^{\prime}=\{\mathbf{y}+q\}_{\mathbf{V},\beta}.$
Moreover
\begin{multline}
  t_\gamma
  -2\pi\sqrt{-1}
  \langle \gamma^\vee,\lambda\rangle
  -\sum_{\beta\in\mathbf{V}_I}
  (t_\beta-2\pi\sqrt{-1}
  \langle\beta^\vee,\lambda\rangle)\langle\gamma^\vee,\mu^{\mathbf{V}}_\beta\rangle
  \\
  \begin{aligned}
    &=
    t_\gamma
    -\sum_{\beta\in\mathbf{V}_I}
    t_\beta\langle\gamma^\vee,\mu^{\mathbf{V}}_\beta\rangle
    -2\pi\sqrt{-1}\Bigl(
    \langle \gamma^\vee,\lambda\rangle
    -\sum_{\beta\in\mathbf{V}_I}
    \langle\beta^\vee,\lambda\rangle\langle\gamma^\vee,\mu^{\mathbf{V}}_\beta\rangle\Bigr)
    \\
    &=
    t_\gamma
    -\sum_{\beta\in\mathbf{V}_I}
    t_\beta\langle\gamma^\vee,\mu^{\mathbf{V}}_\beta\rangle
    -2\pi\sqrt{-1}
    \langle \gamma^\vee,p_{\mathbf{V}_I^\perp}(\lambda)\rangle
  \end{aligned}
\end{multline}
by \eqref{eq:proj}, and
\begin{multline}
    \exp\Bigl(-2\pi\sqrt{-1}\sum_{\beta\in\mathbf{V}_I}
    \langle \beta^\vee,\lambda\rangle
    \{\mathbf{y}+q\}_{\mathbf{V},\beta}\Bigr)
\\
\begin{aligned}
  &=
  \exp\Bigl(-2\pi\sqrt{-1}\sum_{\beta\in\mathbf{V}_I}
  \langle \beta^\vee,\lambda\rangle
    \langle\mathbf{y}+q,\mu_\beta^{\mathbf{V}}\rangle\Bigr)
    \\
    &=
    \exp(-2\pi\sqrt{-1}\langle\mathbf{y}+q,\lambda\rangle)
    \exp\Bigl(2\pi\sqrt{-1}
    \langle\mathbf{y}+q,p_{\mathbf{V}_I^\perp}(\lambda)\rangle\Bigr)
    \\
    &=
    \exp(-2\pi\sqrt{-1}\langle\mathbf{y},\lambda\rangle)
    \exp\Bigl(2\pi\sqrt{-1}
    \langle\mathbf{y}+q,p_{\mathbf{V}_I^\perp}(\lambda)\rangle\Bigr)
  \end{aligned}
\end{multline}
by \eqref{eq:proj} again and the facts
$\langle\beta^\vee,\lambda\rangle, \langle q,\lambda\rangle\in\mathbb{Z}$. 
Collecting these facts, we obtain
$$
F(\mathbf{t},\pi_{I^c}(\mathbf{y});\lambda)=
\exp(-2\pi\sqrt{-1}\langle\mathbf{y},\lambda\rangle)F(\mathbf{t}_I,\mathbf{y},
\lambda;I;\Delta)
$$
and hence
$$
C(\mathbf{k},\pi_{I^c}(\mathbf{y});\lambda)=
\exp(-2\pi\sqrt{-1}\langle\mathbf{y},\lambda\rangle)P(\mathbf{k},\mathbf{y},
\lambda;I;\Delta).
$$
Therefore by \eqref{rev-2} and \eqref{S_lamba=S} we have
\begin{align*}
S_{\lambda}(\mathbf{k},\mathbf{y};I;\Delta)
&=\left(\prod_{\alpha\in\Delta^*}\frac{-(2\pi\sqrt{-1})^{k_{\alpha}}}{k_{\alpha}!}
\right)\\
&\times\exp(-2\pi\sqrt{-1}\langle\mathbf{y},\lambda\rangle)P(\mathbf{k},\mathbf{y},
\lambda;I;\Delta).
\end{align*}
Substituting this into \eqref{S-decomp}, we arrive at the assertion of the theorem.
\end{proof}

%%%%%%%%%%%%%%%%%%%%%%%%%%%%%%%%%%%%%%%%%%%%%%%%%%%%%%%%%%%%%%%%%%%%%%%%%%%%%%%%
\section{Proof of Theorem \ref{thm:main2}}\label{sec9}
%%%%%%%%%%%%%%%%%%%%%%%%%%%%%%%%%%%%%%%%%%%%%%%%%%%%%%%%%%%%%%%%%%%%%%%%%%%%%%%%

In this final section we prove Theorem \ref{thm:main2}.
In what follows, let $\Lh{A}$ denote the $\mathbb{R}$-span of $A$.
Let $n=|\Delta_{I+}|$.
  Fix an order
  $\Delta_{I+}=\{\beta_1,\beta_2,\ldots,\beta_{n}\}$ and
  put $A_j=\{\beta_1,\ldots,\beta_{j}\}$ for $0\leq j\leq n$.
  Note that
  \begin{equation}
    \emptyset=A_0\subset A_1\subset\cdots \subset A_{n}=\Delta_{I+}.
  \end{equation}

  Fix $\mathbf{V}\in\mathscr{V}$ and $q\in Q^\vee/L(\mathbf{V}^\vee)$.
Consider the following condition for $j$:
  \begin{equation}
    \label{eq:condj}
    \beta_k\in\mathbf{V}\cup\Lh{A_{k-1}}\text{ for all $k\leq j$}.
  \end{equation}
Under this condition it follows that, 
in the sequence of nondecreasing linear spaces
  \begin{equation}\label{seq_A}
    \{0\}=[A_0]\subset [A_1]\subset\cdots \subset [A_{n}]=[\Delta_{I+}],
  \end{equation}
if $[A_{k-1}]\subsetneq[A_k]$, then $\beta_k\notin\Lh{A_{k-1}}$ and 
hence $\beta_k\in\mathbf{V}$.

  Let $\mathbf{t}_j=(t_{\gamma})_{\gamma\in \Delta_+\setminus A_j}$
  (hence $\mathbf{t}_0=\mathbf{t}$), 
  $\mathbf{t}_{j,\mathbf{V}}=(t_{\beta})_{\beta\in\mathbf{V}\setminus\Lh{A_j}}$, 
  and
$$
S_{j,\gamma}(t_{\gamma},\mathbf{t}_{j,\mathbf{V}})=t_\gamma-\sum_{\beta\in\mathbf{V}\setminus\Lh{A_j}}
        t_\beta\langle\gamma^\vee,\mu^{\mathbf{V}}_\beta\rangle
        -2\pi\sqrt{-1}
        \sum_{\beta\in\mathbf{V}\cap\Lh{A_j}}\langle\beta^\vee,\lambda\rangle
        \langle\gamma^\vee,\mu^{\mathbf{V}}_\beta\rangle.
$$
 For $j$ satisfying \eqref{eq:condj}, define
  \begin{equation}
    \label{eq:fvqj}
    \begin{split}
      &f_{j}(\mathbf{t}_j;\mathbf{V},q)\\
      &=
      \left(\prod_{\gamma\in \Delta_+\setminus(\mathbf{V}\cup \Lh{A_j})}
      \frac{t_\gamma}
      {S_{j,\gamma}(t_{\gamma},\mathbf{t}_{j,\mathbf{V}})}\right)
      \times
      \left(\prod_{\gamma\in (\Delta_+\cap \Lh{A_j})\setminus(\mathbf{V}\cup A_j)}
      \frac{t_\gamma}
      {t_\gamma-2\pi\sqrt{-1}\langle\gamma^\vee,\lambda\rangle}\right)
      \\
      &\times
      \exp\Bigl(
      2\pi\sqrt{-1}
      \sum_{\beta\in\mathbf{V}\cap\Lh{A_j}}\langle\beta^\vee,\lambda\rangle
      \langle\mathbf{y}+q,\mu^{\mathbf{V}}_\beta\rangle\Bigr)
      \times\left(\prod_{\gamma\in\mathbf{V}\setminus\Lh{A_j}}
      \frac{t_\gamma\exp
        (t_\gamma\{\mathbf{y}+q\}_{\mathbf{V},\gamma})}{e^{t_\gamma}-1}\right)\\
      &=f_{j1}\times f_{j2}\times f_{j3}\times f_{j4},
    \end{split}
  \end{equation}
  say.    If $j$ does not satisfy \eqref{eq:condj}, we let 
  $f_{j}(\mathbf{t}_j;\mathbf{V},q)=0$.    Then
  \begin{equation}\label{F_f0}
    F(\mathbf{t},\mathbf{y};\Delta)
    =
    \sum_{\mathbf{V}\in\mathscr{V}}
    \frac{1}{\abs{Q^\vee/L(\mathbf{V}^\vee)}}
    \sum_{q\in Q^\vee/L(\mathbf{V}^\vee)}
    f_{0}(\mathbf{t};\mathbf{V},q).
  \end{equation}
  
\begin{lemma}\label{residue}
We have
  \begin{equation}
    \label{eq:rec}
    f_{j}(\mathbf{t}_j;\mathbf{V},q)
    =
    \Res_{t_{\beta_j}=2\pi\sqrt{-1}\langle\beta_j^\vee,\lambda\rangle}
    \frac{f_{j-1}(\mathbf{t}_{j-1};\mathbf{V},q)}{t_{\beta_j}},
  \end{equation}
  where the limit procedure of calculating the residue is along some `generic' path
  in the sense of Remark \ref{def_residue}.
\end{lemma}
\begin{proof}
We begin with two preparatory statements in relation to \eqref{eq:condj}.
  First we show that under the condition \eqref{eq:condj}, 
  \begin{equation}\label{statement1}
    \text{either
      $\beta_j\in\mathbf{V}\setminus\Lh{A_{j-1}}$ or
      $\beta_j\in\Lh{A_{j-1}}\setminus\mathbf{V}$ }
  \end{equation}
  holds.  Under the
  condition \eqref{eq:condj} we have $\beta_j\in \mathbf{V}\cup\Lh{A_{j-1}}$.
  Assume $\beta_j\in\mathbf{V}\cap\Lh{A_{j-1}}$.  Then there exists
  $k\leq j-1$ such that $\beta_j\in\Lh{A_{k}}\setminus\Lh{A_{k-1}}$.
  Since $\Lh{A_{k}}\setminus\Lh{A_{k-1}}\neq\emptyset$, we have 
  $\beta_k\in\mathbf{V}$ by the argument just below \eqref{seq_A}.
  Then
  $\beta_j, \beta_k\in (\mathbf{V}\cap \Lh{A_k})\setminus
  (\mathbf{V}\cap \Lh{A_{k-1}})$ and
  $|\mathbf{V}\cap \Lh{A_k}|-|\mathbf{V}\cap \Lh{A_{k-1}}|\geq 2$.
  Since the elements of $\mathbf{V}$ are linearly independent, this
  contradicts to $\dim \Lh{A_{k}}-\dim\Lh{A_{k-1}}\leq 1$.
  Therefore \eqref{statement1} holds.

  Next we show that under the condition \eqref{eq:condj},
  \begin{equation}
    \label{eq:VAj}
    \Lh{A_j}=\Lh{\mathbf{V}\cap[A_j]}
  \end{equation}
  holds.  Since $\Lh{A_j}\supset\Lh{\mathbf{V}\cap[A_j]}$ is trivial,
  we show the opposite inclusion. If $\Lh{A_{k-1}}\subsetneq\Lh{A_{k}}$,
  then $\beta_k\in\mathbf{V}\cap[A_k]$ (again by the argument just below \eqref{seq_A}). 
  Hence
  $\dim[A_j]=|\mathbf{V}\cap[A_j]|$, which implies \eqref{eq:VAj}.
  
Now we start to prove \eqref{eq:rec}.

{\it Case 1}.   The case when $j$ satisfies \eqref{eq:condj}. 
Then by \eqref{statement1} either
      $\beta_j\in\mathbf{V}\setminus\Lh{A_{j-1}}$ or
      $\beta_j\in\Lh{A_{j-1}}\setminus\mathbf{V}$ holds.   
      
   {\it Case 1-1}.     Assume
  $\beta_j\in\mathbf{V}\setminus\Lh{A_{j-1}}$.  
  Then
  $\Lh{A_{j-1}}\subsetneq\Lh{A_j}$ and
  we see that, in the expression \eqref{eq:fvqj} of 
  $f_{j-1}(\mathbf{t}_{j-1};\mathbf{V},q)$, 
  $\beta_j$ appears only in $f_{j-1,1}$ and $f_{j-1,4}$. 
  When $t_{\beta_j}=2\pi\sqrt{-1}\langle\beta_j^\vee,\lambda\rangle$, the singularity
  appears only on the factor corresponding to $\gamma=\beta_j$ in $f_{j-1,4}$, 
  whose contribution to the residue \eqref{eq:rec} is
  \begin{align}
& \Res_{t_{\beta_j}=2\pi\sqrt{-1}\langle\beta_j^\vee,\lambda\rangle}
        \frac{1}{t_{\beta_j}}  \frac{t_{\beta_j}\exp
        (t_{\beta_j}\{\mathbf{y}+q\}_{\mathbf{V},\beta_j})}{e^{t_{\beta_j}}-1}\\ &\qquad\qquad=\exp(2\pi\sqrt{-1}\langle\beta_j^\vee,\lambda\rangle\{\mathbf{y}+q\}_{\mathbf{V},\beta_j})\notag  \\     
&\qquad\qquad=\exp(2\pi\sqrt{-1}\langle\beta_j^\vee,\lambda\rangle\langle\mathbf{y}+q,\mu^{\mathbf{V}}_{\beta_j}\rangle)\notag
  \end{align}
because $\langle\beta_j^\vee,\lambda\rangle\in\mathbb{Z}$, which becomes a member of 
$f_{j3}$.   Therefore the contribution of $f_{j-1,3}f_{j-1,4}$ to the residue
\eqref{eq:rec} is
\begin{align}\label{f3f4}
     & =\exp\Bigl(
      2\pi\sqrt{-1}
      \sum_{\beta\in\mathbf{V}\cap\Lh{A_{j-1}}}\langle\beta^\vee,\lambda\rangle
      \langle\mathbf{y}+q,\mu^{\mathbf{V}}_\beta\rangle\Bigr)\\
     &\times \exp(2\pi\sqrt{-1}\langle\beta_j^\vee,\lambda\rangle\langle\mathbf{y}+q,\mu^{\mathbf{V}}_{\beta_j}\rangle) 
      \prod_{\gamma\in\mathbf{V}\setminus(\Lh{A_{j-1}}\cup\{\beta_j\})}
      \frac{t_\gamma\exp
        (t_\gamma\{\mathbf{y}+q\}_{\mathbf{V},\gamma})}{e^{t_\gamma}-1}\notag\\
 &=\exp\Bigl(
      2\pi\sqrt{-1}
      \sum_{\beta\in\mathbf{V}\cap\Lh{A_j}}\langle\beta^\vee,\lambda\rangle
      \langle\mathbf{y}+q,\mu^{\mathbf{V}}_\beta\rangle\Bigr)
     \prod_{\gamma\in\mathbf{V}\setminus\Lh{A_j}}
      \frac{t_\gamma\exp
        (t_\gamma\{\mathbf{y}+q\}_{\mathbf{V},\gamma})}{e^{t_\gamma}-1}\notag\\
&=f_{j3}f_{j4}.\notag
\end{align}

Next consider the denominator of $f_{j-1,1}$ for 
  $\gamma\in(\Delta_+\cap\Lh{A_j})\setminus(\mathbf{V}\cup
  \Lh{A_{j-1}})$ at
  $t_{\beta_j}=2\pi\sqrt{-1}\langle\beta_j^\vee,\lambda\rangle$.
It is
  \begin{align}
    & S_{j-1,\gamma}(t_{\gamma},\mathbf{t}_{j-1,\mathbf{V}})\Biggr|_{t_{\beta_j}=2\pi\sqrt{-1}\langle\beta_j^\vee,\lambda\rangle}\\
    &=t_\gamma-\Biggl.\sum_{\beta\in\mathbf{V}\setminus\Lh{A_{j-1}}}
    t_\beta\langle\gamma^\vee,\mu^{\mathbf{V}}_\beta\rangle
    \Biggr|_{t_{\beta_j}=2\pi\sqrt{-1}\langle\beta_j^\vee,\lambda\rangle}\notag\\
    &\qquad\qquad\qquad-2\pi\sqrt{-1}
    \sum_{\beta\in\mathbf{V}\cap\Lh{A_{j-1}}}\langle\beta^\vee,\lambda\rangle
    \langle\gamma^\vee,\mu^{\mathbf{V}}_\beta\rangle
    \notag\\
      &=
      t_\gamma-\sum_{\beta\in\mathbf{V}\setminus\Lh{A_{j}}}
      t_\beta\langle\gamma^\vee,\mu^{\mathbf{V}}_\beta\rangle
      -2\pi\sqrt{-1}
      \sum_{\beta\in\mathbf{V}\cap\Lh{A_{j}}}\langle\beta^\vee,\lambda\rangle
      \langle\gamma^\vee,\mu^{\mathbf{V}}_\beta\rangle,\notag
\end{align}
which is further
\begin{align}
      =t_\gamma-2\pi\sqrt{-1}\langle\gamma^\vee,\lambda\rangle
  \end{align}
  because by \eqref{eq:VAj}, $\gamma^\vee$ is a linear combination of
  $\mathbf{V}^\vee\cap\Lh{A_{j}}$, so 
  $ \langle\gamma^\vee,\mu^{\mathbf{V}}_\beta\rangle=0$ for
  $\beta\in\mathbf{V}\setminus\Lh{A_{j}}$ and
  % $\langle\gamma^\vee,\mu^{\mathbf{V}}_\beta\rangle=\delta_{\gamma\beta}$
  % for $\beta\in\mathbf{V}\cap\Lh{A_{j}}$.
  by writing $\gamma^\vee=\sum_{\beta\in\mathbf{V}\cap\Lh{A_{j}}} a_\beta \beta^\vee$, we have
  \begin{equation*}
    \sum_{\beta\in\mathbf{V}\cap\Lh{A_{j}}}\langle\beta^\vee,\lambda\rangle
    \langle\gamma^\vee,\mu^{\mathbf{V}}_\beta\rangle=\langle\gamma^\vee,\lambda\rangle.
  \end{equation*}
  Thus each corresponding member of $f_{j-1,1}$ becomes a member of $f_{j2}$.
   Therefore
\begin{align}\label{1-2-factor}
&\biggl.f_{j-1,1}f_{j-1,2}\biggr|_
{t_{\beta_j}=2\pi\sqrt{-1}\langle\beta_j^\vee,\lambda\rangle}\\
&=\prod_{\gamma\in (\Delta_+\setminus\Lh{A_j}) \setminus(\mathbf{V}\cup \Lh{A_{j-1}})}
      \frac{t_\gamma}
      {S_{j-1,\gamma}(t_{\gamma},\mathbf{t}_{j-1,\mathbf{V}})}
      \times
      \prod_{\gamma\in \mathcal{A}_{j-1}}
      \frac{t_\gamma}
      {t_\gamma-2\pi\sqrt{-1}\langle\gamma^\vee,\lambda\rangle},\notag
\end{align}
where
$$
\mathcal{A}_{j-1}=\big((\Delta_+\cap\Lh{A_j}) \setminus(\mathbf{V}\cup \Lh{A_{j-1}})\big)
   \cup \big((\Delta_+\cap \Lh{A_{j-1}})\setminus(\mathbf{V}\cup A_{j-1})\big).
$$
We note that 
\begin{align}\label{A_j-1}
\mathcal{A}_{j-1}
    =
    (\Delta_+\cap[A_j])\setminus(\mathbf{V}\cup A_{j-1})
\end{align}
holds.    In fact, it is obvious that $\mathcal{A}_{j-1}$ is included in the right-hand
side, while an element of the right-hand side does not belong to 
$\Delta_+\cap\Lh{A_{j-1}}$, it does not belong to $\mathbf{V}\cup\Lh{A_{j-1}}$.
Therefore \eqref{A_j-1} holds, and noting
$\mathbf{V}\cup A_{j-1}=\mathbf{V}\cup A_{j}$, we further find that
$\mathcal{A}_{j-1}=(\Delta_+\cap[A_j])\setminus(\mathbf{V}\cup A_{j})$.
As for the first product on the right-hand side of \eqref{1-2-factor}, we see that
$(\Delta_+\setminus\Lh{A_j}) \setminus(\mathbf{V}\cup \Lh{A_{j-1}})
=\Delta_+\setminus(\mathbf{V}\cup \Lh{A_{j}})$ and
$S_{j-1,\gamma}(t_{\gamma},\mathbf{t}_{j-1,\mathbf{V}})
=S_{j,\gamma}(t_{\gamma},\mathbf{t}_{j,\mathbf{V}})$
(because $\langle\gamma^{\vee},\mu_{\beta_j}^{\mathbf{V}}\rangle=0$ for
$\gamma\in\Delta_+\setminus\Lh{A_j}$).
Therefore \eqref{1-2-factor} can be read as
\begin{align}\label{f1f2}
\biggl.f_{j-1,1}f_{j-1,2}\biggr|_
{t_{\beta_j}=2\pi\sqrt{-1}\langle\beta_j^\vee,\lambda\rangle}
=f_{j1}f_{j2}.
\end{align}
Combining \eqref{f3f4} and \eqref{f1f2} we arrive at \eqref{eq:rec} in the present case.
  
  {\it Case 1-2}. Assume
  $\beta_j\in\Lh{A_{j-1}}\setminus\mathbf{V}$. Then
  $\Lh{A_{j-1}}=\Lh{A_j}$ and we see that $\beta_j$ appears only
  in $f_{j-1,2}$ in the expression \eqref{eq:fvqj} of 
  $f_{j-1}(\mathbf{t}_{j-1};\mathbf{V},q)$. Furthermore
  \begin{equation}
    \Res_{t_{\beta_j}=2\pi\sqrt{-1}\langle\beta_j^\vee,\lambda\rangle}
    \frac{1}{t_{\beta_j}}
    \frac{t_{\beta_j}}
    {t_{\beta_j}-2\pi\sqrt{-1}\langle\beta_j^\vee,\lambda\rangle}=1
  \end{equation}
  and hence \eqref{eq:rec} holds.

{\it Case 2}.  The case when $j$ does not satisfy \eqref{eq:condj}.  
  There exists $k\leq j$ such that
  $\beta_k\notin\mathbf{V}\cup\Lh{A_{k-1}}$.  Take the
  minimum $k$ among them.  If $k\leq j-1$, then
  $f_{j}(\mathbf{t}_j;\mathbf{V},q)=f_{j-1}(\mathbf{t}_{j-1};\mathbf{V},q)=0$
  and \eqref{eq:rec} holds.  Assume $k=j$. 
  Then
  $f_{j-1}(\mathbf{t}_{j-1};\mathbf{V},q)$ is given by \eqref{eq:fvqj} and
  $f_{j}(\mathbf{t}_j;\mathbf{V},q)=0$. 
  In this case $\beta_j$ appears only
  in $f_{j-1,1}$
  %the first factor of the expression \eqref{eq:fvqj}
  %for
  in the expression \eqref{eq:fvqj} of
  $f_{j-1}(\mathbf{t}_{j-1};\mathbf{V},q)$.
  If $\langle\beta_j^\vee,\mu^{\mathbf{V}}_\beta\rangle
  =0$ for all $\beta\in\mathbf{V}\setminus\Lh{A_{j-1}}$,
  then $\beta_j\in\Lh{\mathbf{V}\cap [A_{j-1}]}\subset\Lh{A_{j-1}}$, which 
  contradicts to 
  $\beta_j\notin\mathbf{V}\cup\Lh{A_{j-1}}$.
  Hence at least one of 
  the coefficients of $t_\beta$ for
  $\beta\in\mathbf{V}\setminus\Lh{A_{j-1}}$
  does not vanish in
  \begin{align}
    & S_{j-1,\gamma}(t_{\gamma},\mathbf{t}_{j-1,\mathbf{V}})\Biggr|_{t_{\beta_j}=2\pi\sqrt{-1}\langle\beta_j^\vee,\lambda\rangle}\\
    &=t_{\beta_j}-\sum_{\beta\in\mathbf{V}\setminus\Lh{A_{j-1}}}
    t_\beta\langle\beta_j^\vee,\mu^{\mathbf{V}}_\beta\rangle\notag\\
    &-2\pi\sqrt{-1}
    \sum_{\beta\in\mathbf{V}\cap\Lh{A_{j-1}}}\langle\beta^\vee,\lambda\rangle
    \langle\beta_j^\vee,\mu^{\mathbf{V}}_\beta\rangle\Bigr|_{t_{\beta_j}=2\pi\sqrt{-1}\langle\beta_j^\vee,\lambda\rangle},\notag
  \end{align}
  which implies that the residue is $0$ (by the calculation along some `generic' path, 
  in the sense of Remark \ref{def_residue})
  and so
 \eqref{eq:rec} holds.

  Thus we showed that \eqref{eq:rec} holds in any case.  
\end{proof}

By the repeated use of Lemma \ref{residue}, we obtain
\begin{align}\label{iter_residue}
f_n(\mathbf{t}_n;\mathbf{V},q)=
\Res_{t_{\beta_n}=2\pi\sqrt{-1}\langle\beta_n^\vee,\lambda\rangle}
  \cdots
\Res_{t_{\beta_1}=2\pi\sqrt{-1}\langle\beta_1^\vee,\lambda\rangle}
    \Bigl(\prod_{\alpha\in\Delta_{I+}}\frac{1}{t_{\alpha}}\Bigr)
    f_0(\mathbf{t};\mathbf{V},q).
\end{align}
  
\begin{proof}[Proof of Theorem \ref{thm:main2}]
  Let
  $\Phi_I=\{\beta_k~|~\beta_k\notin\Lh{A_{k-1}}\}\subset\Delta_{I+}$.  Then
  $\Phi_I$ is a linearly independent set with $|\Phi_I|=|I|$. 

  Assume $\Phi_I\subset\mathbf{V}$.
%  Let $\mathbf{V}_I=\mathbf{V}\setminus\Lh{\Delta_{I+}}=\mathbf{V}\setminus\Phi_I$.
  Let $\mathbf{V}_I=\mathbf{V}\setminus\Phi_I$. 
  Since $\Lh{A_n}=\Lh{\Delta_{I+}}$, we see that
  $\mathbf{V}\setminus\Lh{A_n}=\mathbf{V}_I$, 
$$\Delta_{+}\setminus(\mathbf{V}\cup\Lh{A_n})
=\Delta_+\setminus (\mathbf{V}_I\cup\Lh{\Delta_{I+}})
=\Delta^*\setminus \mathbf{V}_I,$$
$$(\Delta_+\cap\Lh{A_n})\setminus(\mathbf{V}\cup A_n)
=\Delta_{I+}\setminus(\mathbf{V}\cup\Delta_{I+})=\emptyset.$$ 
% Moreover if
% $\Phi_I\subset\mathbf{V}$, then
Further we have
$$\mathbf{V}\cap\Lh{A_n}=\mathbf{V}\setminus\mathbf{V}_I=\Phi_I.$$
Therefore \eqref{eq:fvqj} for $j=n$ implies  
  \begin{equation}
    \begin{split}
      &f_{n}(\mathbf{t}_n;\mathbf{V},q)\\
      &=
      \prod_{\gamma\in \Delta^*\setminus\mathbf{V}_I}
      \frac{t_\gamma}
      {t_\gamma-\sum_{\beta\in\mathbf{V}_I}
        t_\beta\langle\gamma^\vee,\mu^{\mathbf{V}}_\beta\rangle
        -2\pi\sqrt{-1}
        \sum_{\beta\in\Phi_I}\langle\beta^\vee,\lambda\rangle
        \langle\gamma^\vee,\mu^{\mathbf{V}}_\beta\rangle}
      \\
      &\qquad\times
      \exp\Bigl(
      2\pi\sqrt{-1}
      \sum_{\beta\in\Phi_I}\langle\beta^\vee,\lambda\rangle
      \langle\mathbf{y}+q,\mu^{\mathbf{V}}_\beta\rangle\Bigr)
      \prod_{\gamma\in\mathbf{V}_I}
      \frac{t_\gamma\exp
        (t_\gamma\{\mathbf{y}+q\}_{\mathbf{V},\gamma})}{e^{t_\gamma}-1}.
    \end{split}
  \end{equation}

  If $\Phi_I\not\subset\mathbf{V}$, then for
  $\beta_k\in\Phi_I\setminus\mathbf{V}$, we have
  $\beta_k\notin\mathbf{V}\cup\Lh{A_{k-1}}$ and hence
  $f_{n}(\mathbf{t}_n;\mathbf{V},q)=0$.
%  $f_{\mathbf{V}q|\Delta_{I+}|}(\mathbf{t})=0$.  
  
  Applying Lemma \ref{lm:indept_bases} to the right-hand side of \eqref{eq:exp_F}, 
 and noting that there is a one-to-one correspondence between 
  \begin{equation}
    \{\mathbf{V}\in\mathscr{V}~|~\Phi_I\subset\mathbf{V}\}
    \qquad\text{and}\qquad
    \{\mathbf{V}_I~|~\mathbf{V}_I\cup\Psi_I\in\mathscr{V}_I\},
  \end{equation}
  we obtain
  \begin{equation}
    F(\mathbf{t}_I,\mathbf{y},\lambda;I;\Delta)
    =
    \sum_{\mathbf{V}\in\mathscr{V}}
    \frac{1}{\abs{Q^\vee/L(\mathbf{V}^\vee)}}
    \sum_{q\in Q^\vee/L(\mathbf{V}^\vee)}
    f_{n}(\mathbf{t}_n;\mathbf{V},q).
  \end{equation}
  by using $\mathbf{t}_n=\mathbf{t}_I$.
  Combining this with \eqref{iter_residue}, and using \eqref{F_f0}, we finally
  obtain the formula \eqref{egregium} stated in Theorem \ref{thm:main2}.
  
So far we have discussed under the fixed choice of the order 
$\beta_1,\ldots,\beta_n$ of the elements of $\Delta_{I+}$. 
The right-hand side of \eqref{iter_residue} apparently depends on this choice.
However the left-hand side of \eqref{egregium}, defined by \eqref{eq:exp_F}, 
does not depend on the choice of
the order.    
So is the right-hand side of \eqref{egregium}.   
This implies the claim mentioned in Remark \ref{def_residue}, and
hence the proof of Theorem \ref{thm:main2} is thus complete.
    
\end{proof}

\end{document}